\documentclass[11pt,twoside]{amsart}
\usepackage{latexsym,amssymb,amsmath}
\usepackage[all]{xy}
\usepackage{tikz,pgfplots}
\usepackage{extpfeil}
\usepackage{afterpage}
\usepackage{float}
\usepackage{hyperref}

\numberwithin{equation}{section}
\newtheorem{theorem}{Theorem}[section]
\newtheorem{lemma}[theorem]{Lemma}
\newtheorem{proposition}[theorem]{Proposition}
\newtheorem{corollary}[theorem]{Corollary}

\theoremstyle{definition}
\newtheorem{definition}[theorem]{Definition} 
\newtheorem{procedure}[theorem]{Procedure} 
\newtheorem{remark}[theorem]{Remark}
\newtheorem{example}[theorem]{Example}

\newcommand{\cL}{\mathcal{L}}

\begin{document}

\title[The dual of an evaluation code]
{The dual of an evaluation code} 
\author{Hiram H. L\'opez}
\address[Hiram H. L\'opez]{Department of Mathematics and Statistics\\ Cleveland State University\\
Cleveland, OH USA}
\email{h.lopezvaldez@csuohio.edu}

\thanks{
The first author was partially supported by an AMS--Simons Travel Grant.
The third author was supported by SNI, Mexico.}

\author{Ivan Soprunov}
\address[Ivan Soprunov]{Department of Mathematics and Statistics\\ Cleveland State University\\
Cleveland, OH USA}
\email{i.soprunov@csuohio.edu}

\author[R. H. Villarreal]{Rafael H. Villarreal}
\address{
Departamento de
Matem\'aticas\\
Centro de Investigaci\'on y de Estudios
Avanzados del
IPN\\
Apartado Postal
14--740 \\
07000 Mexico City, Mexico
}
\email{vila@math.cinvestav.mx}

\begin{abstract} 
The aim of this work is to study the dual and the algebraic dual of an evaluation 
code using standard monomials and indicator functions. We show that the dual of an
evaluation code is the evaluation code of the algebraic dual.
We develop an algorithm for computing a 
basis for the algebraic dual. Let $C_1$ and
$C_2$ be linear codes spanned by standard monomials. We give a combinatorial condition
for the monomial equivalence of $C_1$ and the dual $C_2^\perp$. Moreover, we
give an explicit description of a generator matrix of $C_2^\perp$ in terms of that of $C_1$ and
coefficients of indicator functions. For Reed--Muller-type 
codes we give a duality criterion in terms of the v-number and the Hilbert function of a
vanishing ideal. As an application, we provide an explicit duality for Reed--Muller-type 
codes corresponding to Gorenstein ideals.
In addition, when the evaluation code is monomial and the set of evaluation points is a 
degenerate affine space, we classify when
the dual is a monomial code. 
\end{abstract}

\maketitle 

\section{Introduction}\label{intro-section}
Let $S=K[t_1,\ldots,t_s]=\bigoplus_{d=0}^\infty S_d$ be a polynomial ring over a finite
field $K=\mathbb{F}_q$ with the standard grading and let
$X=\{P_1,\ldots,P_m\}$, $|X|\geq 2$, be a set of distinct points in the affine space
$\mathbb{A}^s:=K^s$. The \textit{evaluation map}, denoted ${\rm ev}$,
is the $K$-linear map given by  
$${\rm ev}\colon S\rightarrow K^{m},\quad
f\mapsto\left(f(P_1),\ldots,f(P_m)\right).
$$ 
\quad The kernel of ${\rm ev}$, denoted $I=I(X)$, is the 
\textit{vanishing ideal} of $X$ consisting of the polynomials of $S$
that vanish at all points of $X$. This map induces an isomorphism of
$K$-linear spaces between $S/I$ and $K^m$. If $f\in
S$, we denote the set of zeros of $f$ in $X$ by $V_{X}(f)$. Let $\mathcal{L}$ be a
linear subspace of $S$ of finite dimension. The image of $\mathcal{L}$
under the evaluation map, denoted $\mathcal{L}_X$, is called an
\textit{evaluation code} on $X$ \cite{toric-codes,stichtenoth,tsfasman}.  

The basic \textit{parameters} of the linear 
code $\mathcal{L}_X$ that we consider are:
\begin{itemize}
\item[(a)] \textit{length}: $m=|X|$,
\item[(b)] \textit{dimension}: $k=\dim_K(\mathcal{L}_X)$, and 
\item[(c)] \textit{minimum distance}: 
$\delta(\mathcal{L}_X)=\min\{|X\setminus V_{X}(f)|\ \colon 
f\in\mathcal{L}\setminus I\}$.   
\end{itemize}
\quad The dual of $\mathcal{L}_X$, denoted $(\mathcal{L}_X)^\perp$, is the
set of all $\alpha\in K^m$ such that $\langle
\alpha,\beta\rangle=0$ for all $\beta\in\mathcal{L}_X$, where
$\langle\ ,\, \rangle$ is the ordinary inner product in $K^m$. 
The dual of $\mathcal{L}_X$ is an $[m,m-k]$ linear code
\cite[Theorem~1.2.1]{Huffman-Pless}. 
The aim of 
this paper is to study $(\mathcal{L}_X)^\perp$ by fixing a graded monomial
order on $S$ and using the information encoded in the quotient ring
$S/I$ and in the linear space $\mathcal{L}$. 

Let $\prec$ be a graded monomial order on $S$, that is, monomials are first compared
by their total degrees \cite[p.~54]{CLO}. The monomials of $S$ are denoted
$t^c:=t_1^{c_1}\cdots t_s^{c_s}$, $c=(c_1,\dots,c_s)$ in $\mathbb{N}^s$, where
$\mathbb{N}=\{0,1,\ldots\}$. We denote the initial monomial of a
non-zero polynomial $f\in S$ by ${\rm in}_\prec(f)$ and the initial
ideal of $I$ by ${\rm in}_\prec(I)$. A subset $\mathcal{G}=\{g_1,\ldots, g_n\}$ of $I$ is called a 
{\it Gr\"obner basis\/} of $I$ if ${\rm
in}_\prec(I)=({\rm in}_\prec(g_1),\ldots,{\rm in}_\prec(g_n))$. A monomial $t^a$ is called a 
\textit{standard monomial} of $S/I$, with respect 
to $\prec$, if $t^a\notin{\rm in}_\prec(I)$. 
The \textit{footprint} of $S/I$ or {\it Gr\"obner \'escalier\/} of $I$, denoted
$\Delta_\prec(I)$, is the finite set of all standard monomials of $S/I$. The
footprint has been used in connection with many kinds of codes and
their basic parameters 
\cite{carvalho,geil-2008,geil-hoholdt,geil-pellikaan,rth-footprint,Pellikaan,toric-codes}. 

If $\mathcal{A}\subset S$, the $K$-linear subspace 
of $S$ spanned by $\mathcal{A}$ is denoted by $K\mathcal{A}$.  
The linear code $\mathcal{L}_X$ is called a \textit{standard
evaluation code} on $X$ relative to $\prec$ if 
$\mathcal{L}$ is a linear subspace 
of $K\Delta_\prec(I)$.
A polynomial $f$ is called
a \textit{standard polynomial} of $S/I$ if $f\neq 0$ and $f$ is in
$K\Delta_\prec(I)$. 
As the field $K$ and the footprint $\Delta_\prec(I)$ are finite,
there are only a finite number of standard 
polynomials. Any evaluation code $\mathcal{L}_X$ on $X$ can be regarded as a
standard evaluation code on $X$ after a suitable transformation of
a generating set for $\mathcal{L}$ \cite{toric-codes}
(Proposition~\ref{transforming-new}). Furthermore, 
given $\mathcal{L}\subset S$ and a monomial order $\prec$, there exists a unique linear subspace
$\widetilde{\mathcal{L}}$ of $K\Delta_\prec(I)$ such 
that $\widetilde{\mathcal{L}}_X=\mathcal{L}_X$
(Corollary~\ref{unique-standard}, Example~\ref{8points-in-A3}). 
We call $\widetilde{\mathcal{L}}$ 
the \textit{standard function space} of $\mathcal{L}_X$. In principle, the basic 
parameters of $\mathcal{L}_X$ can be computed once we determine finite generating
sets for $\widetilde{\mathcal{L}}$ and $I$
\cite{toric-codes}.

Following \cite[p.~16]{Bras-Amoros-O'Sullivan}, let $\varphi$ be the $K$-linear map given by 
$$
\varphi\colon S\rightarrow K,\quad f\mapsto f(P_1)+\cdots+f(P_m).
$$
\quad The kernel of $\varphi$ is a linear 
subspace of $S$ and $S/{\rm ker}(\varphi)\simeq K$. The 
linear subspace of $S$ of all $g\in S$ such that $g\mathcal{L}\subset{\rm
ker}(\varphi)$ is denoted by $({\rm ker}(\varphi)\colon \mathcal{L})$.
The \textit{algebraic dual} of $\mathcal{L}_X$ relative to $\prec$,
denoted $\mathcal{L}^\perp$, 
is the $K$-linear subspace of $S$ given by  
$$\mathcal{L}^\perp:=({\rm ker}(\varphi)\colon
\mathcal{L})\textstyle\bigcap K\Delta_\prec(I),$$
we will also call $\mathcal{L}^\perp$ the \textit{dual} of 
$\mathcal{L}$. The dual $\mathcal{L}^\perp$ is isomorphic to
$(\mathcal{L}^\perp)_X$ (Lemma~\ref{apr26-20}).

Families of linear codes that are closed under taking duals include
generalized 
toric codes \cite[Proposition~3.5]{Bras-Amoros-O'Sullivan},
\cite[Theorem~6]{Ruano}, monomial evaluation codes over the affine
space $\mathbb{A}^s$ that are divisor closed \cite[Proposition~2.4,
Remark~2.5]{Bras-Amoros-O'Sullivan},  
$q$-ary Reed--Muller codes
\cite[Theorem~2.2.1]{delsarte-goethals-macwilliams},
\cite[Remark~4.7]{Pellikaan}, projective Reed--Muller-type
codes over complete intersections \cite[Theorem~2]{sarabia7}, and 
algebraic geometry codes \cite[Theorem~2.2.10]{stichtenoth}. In
these cases duality formulas for the
respective dual codes are given. 

The next result gives a formula for the dual of
$\mathcal{L}_X$ 
in terms of its algebraic dual.

\noindent \textbf{Theorem~\ref{formula-dual}.}\textit{
$(\mathcal{L}_X)^\perp$ is the standard evaluation code
$(\mathcal{L}^\perp)_X$ on $X$ relative 
to $\prec$.
}

A subspace $L$ of $S$ is called
a \textit{monomial space} of $S$ if $L=K\{t^{a_1},\ldots,t^{a_k}\}$ for some
$t^{a_1},\ldots,t^{a_k}$. We say that $\mathcal{L}_X$ is a 
\textit{monomial code} if $\mathcal{L}$ is a monomial space of $S$, and we
say that $\mathcal{L}_X$ is a 
\textit{standard monomial code} if the standard function
space $\widetilde{\mathcal{L}}$ of $\mathcal{L}_X$ is a monomial space
of $S$ (cf. \cite[Definition~1.1]{Lopez-Matthews-Soprunov}). 
If $I$ is a binomial ideal, that is, $I$ is generated by elements of
the form  $t^a-t^b$, and $\mathcal{L}$ is a monomial space of $S$, then
$\mathcal{L}_X$ is a standard monomial code 
(Proposition~\ref{binomial-standard}). 
As an application of Theorem~\ref{formula-dual} we obtain an 
effective criterion for verifying whether or not 
the dual of an evaluation code is a monomial code
(Proposition~\ref{monomial-criterion}, Procedure~\ref{8points-in-A3-procedure}).  

The formula of Theorem~\ref{formula-dual} can be used to compute a
generating set for the algebraic dual of $\mathcal{L}_X$. We show an 
effective algorithm, based on Gaussian elimination, to compute a $K$-basis of any linear subspace of
$S$ of finite dimension (Theorem~\ref{dim-algo}). 
This algorithm can be used to compute a $K$-basis for the
algebraic dual $\mathcal{L}^\perp$ of $\mathcal{L}_X$
 and also for the standard function space $\widetilde{\mathcal{L}}$ of
 $\mathcal{L}_X$ (Examples~\ref{8points-in-A3}, \ref{7points-in-A3}--\ref{5points-in-A2},
 Procedure~\ref{8points-in-A3-procedure}).

In Section~\ref{v-number-section} we introduce
and study the v-number of $I$ \cite{min-dis-generalized}, and the
indicator functions of $X$ 
that are used in coding
theory \cite{min-dis-generalized,sorensen}, Cayley--Bacharach 
schemes \cite{geramita-cayley-bacharach}, and interpolation problems
\cite{cocoa-book}. 
As is seen later in the
introduction these notions are used as devices to study the duality
of standard monomial codes, as well as the asymptotic behavior of 
the minimum distance and the duality of Reed--Muller-type codes. 

Let ${\rm
Ass}(I)=\{\mathfrak{p}_1,\ldots,\mathfrak{p}_m\}$ be the 
set of \textit{associated primes} of 
$I$, that is, $\mathfrak{p}_i$ is the vanishing 
ideal $I_{P_i}$ of $P_i$ and $I=\bigcap_{i=1}^m\mathfrak{p}_i$ is the primary
decomposition of $I$ (Lemma~\ref{primdec-ixx}). The v-\textit{number}
of the ideal $I$ at $\mathfrak{p}_i$, denoted ${\rm v}_{\mathfrak{p}_i}(I)$, is
given by
$$
{\rm v}_{\mathfrak{p}_i}(I):=\mbox{min}\{d\geq 0 \mid
\exists\, 0\neq f\in S,\ \deg(f)=d,\mbox{ with }(I\colon
f)=\mathfrak{p}_i\},
$$
where $(I\colon f):=\{g\in S\, \vert\, gf\subset I \}$ is a
\textit{colon ideal}, and the {\rm v}-\textit{number} of the ideal $I$, denoted
${\rm v}(I)$, is given by ${\rm v}(I):={\rm min}\{{\rm
v}_{\mathfrak{p}_i}(I)\}_{i=1}^m$.  
The notion of v-number is related to indicator
functions as we now
explain.  A polynomial $f$ in $S$ is called an
\textit{indicator function} for $P_i$ if $f(P_i)\neq 0$ and $f(P_j)=0$ if
$j\neq i$ \cite{sorensen}. Indicator functions can be
computed using \cite[Corollary~6.3.11]{cocoa-book}. The v-number of
$I$ at $\mathfrak{p}_i$ is the least degree of 
an indicator function for $P_i$ (Lemma~\ref{if-new}). 

For an ideal $M\neq 0$ of $S/I$, we define 
$\alpha(M)$ to be the minimum degree of the non-zero elements of $M$.
To compute the v-number using {\it Macaulay\/}$2$ \cite{mac2} 
(Example~\ref{Hiram-example}), we give the following description
for the v-number of $I$ at $\mathfrak{p}_i$
(Proposition~\ref{lem:vnumber}):
$${\rm
v}_{\mathfrak{p}_i}(I)=\alpha\!\left((I\colon\mathfrak{p}_i)/{I}\right)\text{ for
all $i$}.
$$
This computation, along with other coding theory tools, is implemented in \cite{package}.

For each point $P_i$ there exists a unique indicator function $f_i$ for $P_i$ in
$K\Delta_\prec(I)$ satisfying $f_i(P_i)=1$, furthermore the degree of $f_i$ is ${\rm
v}_{\mathfrak{p}_i}(I)$, and $F=\{f_1,\ldots,f_m\}$ is a $K$-basis for 
$K\Delta_\prec(I)$ (Proposition~\ref{indicator-function-prop}(a)).
 We call $f_i$ the $i$-th \textit{standard indicator function} for
 $P_i$ and call $F$ the set of \textit{standard indicator
functions} for $X$. As a byproduct we obtain an algebraic method to compute the set
$F$ (Remark~\ref{nov23-20}, Example~\ref{Hiram-example},
Procedure~\ref{8points-in-A3-procedure}). We give the following
formula
\begin{align*}
&{\rm ker}(\varphi)=K\{f_i-f_m\}_{i=1}^{m-1}+I,
\end{align*}
for the kernel of the map $\varphi$ that was used earlier to define
$\mathcal{L}^\perp$ (Proposition~\ref{indicator-function-prop}). 

Given a subset 
$\Gamma\subset\Delta_\prec(I)$, let $\cL(\Gamma)$ be the $K$-span of the set
of all $t^a\in\Gamma$. Then $\cL(\Gamma)_X$ 
is called the {\it standard monomial code} of $\Gamma$.
Consider two standard monomial codes $\cL(\Gamma_1)_X$ and $\cL(\Gamma_2)_X$ for
some $\Gamma_1,\Gamma_2\subset\Delta_\prec(I)$. 
We give a combinatorial condition for the monomial equivalence of
$\cL(\Gamma_1)_X$ and ${\cL(\Gamma_2)_X}^\perp$. For convenience we
recall the definition of this notion. 
We say that two linear codes $C_1,C_2$ in $K^m$
are \textit{monomially equivalent} if there is $\beta=(\beta_1,\ldots,\beta_m)$
in $K^m$ such that $\beta_i\neq 0$ for all $i$ and 
$$C_2 = \beta\cdot C_1=\{\beta\cdot c \mid c\in C_1\},$$
where $\beta\cdot c$ is the vector given by
$(\beta_1c_1,\ldots,\beta_mc_m)$ for 
$c=(c_1,\ldots,c_m)\in C_1$. 

To state the main result of
Section~\ref{S:duality-standard-monomial-codes} we will need the
following definition.  We say a standard monomial 
$t^e\in\Delta_\prec(I)$ is {\it essential} if it appears in each
standard indicator function of $X$.

\noindent \textbf{Theorem~\ref{T:combinatorial-condition}.}\textit{ 
Let $t^e$ be 
essential. Then for any $\Gamma_1,\Gamma_2\subset\Delta_\prec(I)$
satisfying  
\begin{enumerate}
\item $|\Gamma_1|+|\Gamma_2|=|X|$,
\item $t^e$ does not appear in the reduction of $u_1u_2$ modulo $I$ for every
$u_1\in\Gamma_1$ and $u_2\in\Gamma_2$,
\end{enumerate}
we have 
$\beta\cdot\cL(\Gamma_1)_X={\cL(\Gamma_2)_X}^\perp,$
for some $\beta=(\beta_1,\dots, \beta_m)\in K^m$ such that
$\beta_i\neq 0$ for all $i$. Moreover,
$\beta_i$ is the coefficient of $t^e$ in the $i$-th standard indicator 
function~$f_i$, for all $i$. 
}

Given an integer
$d\geq 0$, we let $S_{\leq 
d}=\bigoplus_{i=0}^dS_i$ be the $K$-linear subspace of $S$ of all 
polynomials of degree at most $d$ and let $I_{\leq d}=I\bigcap S_{\leq d}$.  
We set $S_{\leq -1}=\{0\}$, by convention.
The function
$$
H_I^a(d):=\dim_K(S_{\leq d}/I_{\leq d}),\ \ \ d=-1,0,1,2,\ldots
$$
is  called the \textit{affine Hilbert function} of $S/I$. In particular,
$H_I^a(-1)=0$. The \textit{regularity index} of
$H_I^a$, denoted $r_0={\rm reg}(H_I^a)$, is the 
least integer $\ell\geq 0$ such that $H_I^a(d)=|X|$ for $d\geq\ell$
(Proposition~\ref{behavior-hilbert-function}). Note that $r_0\geq 1$
because $|X|\geq 2$. 

If $\mathcal{L}$ is equal to $S_{\leq d}$, then the resulting evaluation
code $\mathcal{L}_X$ is called a \textit{Reed-Muller-type code} of degree $d$ on $X$
\cite{duursma-renteria-tapia,GRT} and is denoted by $C_X(d)$. 

The 
minimum distance of $C_X(d)$ is simply 
denoted by $\delta_X(d)$. 
The v-number of $I$ is related to the asymptotic behavior of
$\delta_X(d)$ for $d\gg 0$. By
Proposition~\ref{behavior-hilbert-function}, there $n$ in $\mathbb{N}$ such that 
$$
|X|=\delta_X(0)>\delta_X(1)>\cdots>\delta_X(n-1)>\delta_X(n)=\delta_X(d)=1\
\mbox{ for }\ d\geq n.
$$
\quad The number $n$, denoted ${\rm reg}(\delta_X)$, is called the 
\textit{regularity index} of $\delta_X$. By the 
Singleton bound \cite[p.~71]{Huffman-Pless}, one has the inequality ${\rm
reg}(\delta_X)\leq{\rm reg}(H_I^a)$. Using indicator functions we
prove that ${\rm v}(I)$ is equal to ${\rm reg}(\delta_X)$ 
(Proposition~\ref{tuesday-afternoon}), and consequently using
Proposition~\ref{lem:vnumber} we can 
compute ${\rm reg}(\delta_X)$ with {\it Macaulay\/}$2$ \cite{mac2}
(Example~\ref{8points-in-A3}). 

In Section~\ref{duality-criterion-rm} we give a duality criterion for
the monomial equivalence of the linear codes $C_X(d)$ and $C_X(r_0-d-1)^\perp$ 
for $-1\leq d\leq r_0$, where 
$r_0={\rm reg}(H_I^a)$. 
As
$\dim_K(C_X(d))=H_I^a(d)$, a necessary condition for this equivalence
is the equality 
$$H_I^a(d)+H_I^a(r_0-d-1)=|X|\ \mbox{ for }\ -1\leq d\leq r_0.$$
\quad In 
Section~\ref{prelim-section} we
characterize this equality in terms of the symmetry of
$h$-vectors and the symmetry of the function 
$\psi(d)=|\Delta_\prec(I)\bigcap S_d|$ (Proposition~\ref{duality-hilbert-function}). 

We come to one of our main results. 

\noindent \textbf{Theorem~\ref{duality-criterion}.} (Duality
criterion)\textit{  Let $r_0={\rm reg}(H_I^a)$. The following are equivalent.
\begin{enumerate} 
\item[(a)] $C_X(d)$ is monomially equivalent to $C_X(r_0-d-1)^\perp$  for 
$-1\leq d\leq r_0$.  
\item[(b)] $H_I^a(d)+H_I^a(r_0-d-1)=|X|$ for $-1\leq d\leq r_0$ 
and $r_0={\rm v}_{\mathfrak{p}}(I)$ for $\mathfrak{p}\in{\rm Ass}(I)$.
\item[(c)] There is 
$g\in K\Delta_\prec(I)$ such that $g(P_i)\neq 0$ for all $i$ and
\begin{equation*}
C_X(r_0-d-1)^\perp=(g(P_1),\ldots,g(P_m))\cdot C_X(d)\ \text{ for
\ $-1\leq d\leq r_0$}.
\end{equation*}
\end{enumerate}
} 
The standard polynomial $g$ of
part (c) is unique up to 
multiplication by a scalar from $K^*$, where $K^*:=K\setminus\{0\}$. If $F=\{f_1,\ldots,f_m\}$ is the
unique set of standard indicator functions for $X$, then $g$ is equal to 
$\sum_{i=1}^m{\rm lc}(f_i)f_i$ (see Example~\ref{Ivan-Hiram} for an illustration). 
The value of $g$ at $P_i$ is ${\rm lc}(f_i)$, the leading coefficient of $f_i$. We will use this
criterion to show duality for some
interesting families and recover some known results. 
The condition $r_0={\rm v}_\mathfrak{p}(I)$ for
$\mathfrak{p}\in{\rm Ass}(I)$ that appears in the duality criterion
defines a Cayley--Bacharach scheme (CB-scheme) in the projective case
when $K$ is an infinite field 
\cite[Definition~2.7]{geramita-cayley-bacharach}, and is related to Hilbert functions. 

Gorenstein and complete intersection ideals---and some of their
properties---are introduced in
Section~\ref{prelim-section}. If $I$ is a complete intersection
generated by a Gr\"obner basis with $s$ elements, then
the ideal $I$ is Gorenstein (Corollary~\ref{duality-hilbert-gorenstein}(c)). The
converse is not true (Example~\ref{Ivan-Hiram}). 
If $I$ is Gorenstein, then $H_I^a(d)+H_I^a(r_0-d-1)=|X|$ for $-1\leq
d\leq r_0$ (Corollary~\ref{duality-hilbert-gorenstein}(a)). 
This result, together with the next theorem, shows 
that the 
combinatorial condition of
Theorem~\ref{T:combinatorial-condition} and the conditions of
Theorem~\ref{duality-criterion}(b) are satisfied 
when $I$ is
a Gorenstein ideal. 

\noindent \textbf{Theorem~\ref{gorenstein-vnumber}.}\textit{ 
Let $F=\{f_1,\ldots,f_m\}$ be the set of standard indicator
functions for $X$. If $I$ is Gorenstein, then  ${\rm
reg}(H_I^a)={\rm v}_{\mathfrak{p}}(I)$ for $\mathfrak{p}\in{\rm
Ass}(I)$ and $\mathrm{in}_\prec(f_i)=\mathrm{in}_\prec(f_m)$ for all $i$.
}

The following result includes the family of Reed--Muller-type codes over complete
intersections and in particular---since vanishing ideals of Cartesian
sets are complete intersections generated by a Gr\"obner basis with
$s$ elements \cite[Lemma~2.3]{cartesian-codes}---we recover the duality theorems for  
affine Cartesian codes given in \cite[Theorem~5.7]{GHWCartesian} and 
\cite[Theorem~2.3]{Lopez-Manganiello-Matthews}.

\noindent \textbf{Corollary~\ref{gorenstein-codes}.}\textit{ 
Let $r_0$ be the regularity index of $H_I^a$. If $I$ is a Gorenstein
ideal, then there is 
$g\in K\Delta_\prec(I)$ such that $g(P_i)\neq 0$ for all $i$ and
\begin{equation*}
(g(P_1),\ldots,g(P_m))\cdot C_X(r_0-d-1)=C_X(d)^\perp\ \text{ for
\ $-1\leq d\leq r_0$}.
\end{equation*}
}
\quad As an application, we produce self-dual Reed--Muller-type 
codes when $I$ is Gorenstein, ${\rm char}(K)=2$, and $r_0$ is odd
(Corollary~\ref{self-dual-gorenstein-codes}). 

\quad In Section~\ref{dual-monomial-section} we give an explicit
description for the
algebraic dual of $\mathcal{L}_T$ when $\mathcal{L}$ is a monomial 
space of $S$ and $T=\{P_1,\ldots,P_m\}$ is the set of points in a degenerate
torus (Proposition~\ref{dual-toric-degenerate}, Example~\ref{example1}). In this case the
vanishing ideal of $T$ is a complete
intersection binomial ideal and, by
Proposition~\ref{binomial-standard}, $\mathcal{L}_T$ is a standard
monomial code. Let $T=(K^*)^s$ be a torus in $\mathbb{A}^s$ and let
$\mathcal{L}_T$ be a \textit{generalized toric code} on $T$ 
\cite{Little,Ruano,Soprunov}, that is, $\mathcal{L}$ is a monomial
space of $S$. Bras-Amor\'os and O'Sullivan \cite[Proposition~3.5]{Bras-Amoros-O'Sullivan} 
and independently Ruano \cite[Theorem~6]{Ruano} compute the dual of
$\mathcal{L}_T$ and show that the dual of $\mathcal{L}_T$ 
is a generalized toric code. As an application we
recover these results (Corollary~\ref{dual-toric-coro}).

The rest of this paper is devoted to study the dual of monomial codes
on a degenerate affine space. 
Let $K=\mathbb{F}_q$ be a finite field of characteristic $p$, let $A_1,\ldots,A_s$ be  
subgroups of the multiplicative group $K^*$ of 
$K$, let $B_i$ be the set $A_i\bigcup\{0\}$ for $i=1,\ldots,s$, 
let  
$$\mathcal{X}:=B_1\times\cdots\times B_s,$$
and let $\mathcal{L}_\mathcal{X}$
be a monomial code on $\mathcal{X}$. The set $\mathcal{X}$ is called 
a \textit{degenerate affine space}. In this case the vanishing ideal
$I=I(\mathcal{X})$ is a complete intersection binomial ideal. By 
Proposition~\ref{binomial-standard}, we may assume
that $\mathcal{L}$ is the standard function space
$\widetilde{\mathcal{L}}$ of $\mathcal{L}_\mathcal{X}$
and that
$\mathcal{A}=\{t^{a_1},\ldots,t^{a_k}\}\subset\Delta_\prec(I)$ is a
$K$-basis for $\mathcal{L}$ with $a_i=(a_{i,1},\ldots,a_{i,s})$ for $i=1,\ldots,k$. 
Following
\cite[p.~16]{Bras-Amoros-O'Sullivan}, we say that 
$\mathcal{A}$ is \textit{divisor-closed} if $t^a\in\mathcal{A}$ whenever $t^a$ divides
a monomial in $\mathcal{A}$. To classify when the dual of
$\mathcal{L}_\mathcal{X}$ is a standard monomial code, we introduce a
weaker notion than divisor-closed that we call \textit{weakly
divisor-closed} 
(Definition~\ref{weakly-divisor-closed-def}). The order of the
multiplicative monoid $B_i$ is denoted by $e_i$ and  
the order of $A_i$ is denoted by $d_i$ for $i=1,\ldots,s$. For use below we set 
\begin{equation*}
t^{b_i}=t_1^{b_{i,1}}\cdots
t_s^{b_{i,s}}:=\displaystyle\textstyle\prod_{j=1}^st_j^{d_j-a_{i,j}},
\end{equation*}
for $i=1,\ldots,k$ and $\mathcal{B}:=\{t^{b_1},\ldots,t^{b_k}\}$. 
Note that $(\mathcal{L}^\perp)_\mathcal{X}$ is a standard monomial
code if and only if $\mathcal{L}^\perp$ is a monomial space of $S$ because
the standard function space of $(\mathcal{L}^\perp)_\mathcal{X}$ is
$\mathcal{L}^\perp$. 

We come to another of our main results.

\noindent \textbf{Theorem~\ref{dual-affine-degenerate}.}\textit{
Let $K$ be a field of characteristic $p$ and $\mathcal{X}$ a degenerate affine space as above.
Assume that $\gcd(p,e_i)=p$, where $e_i=|B_i|$,
for $i=1,\ldots,s$. The following are equivalent.
\begin{enumerate}
\item[(a)] $\mathcal{A}=\{t^{a_1},\ldots,t^{a_k}\}$ is 
weakly divisor-closed. 
\item[(b)] $\mathcal{L}^\perp=K(\Delta_\prec(I)\setminus\mathcal{B})$. 
\item[(c)] $(\mathcal{L}_\mathcal{X})^\perp$ is
a standard monomial code on $\mathcal{X}$. 
\end{enumerate}
}

Let $\mathcal{L}_\mathcal{X}$ be a monomial standard evaluation code on
$\mathcal{X}=K^s$. Then $\mathcal{L}$ is generated by a
subset $\mathcal{A}$ of $\Delta_\prec(I)$. 
Bras-Amor\'os and O'Sullivan \cite[Proposition~2.4, Remark~2.5]{Bras-Amoros-O'Sullivan} 
 compute the dual of $\mathcal{L}_\mathcal{X}$ when $\mathcal{A}$ is divisor
 closed. As an application we recover this result
 (Corollary~\ref{dual-affine-coro}).

\quad In the next result we determine the dual of $K(S_{\leq
d}\bigcap \Delta_\prec(I))$.

\noindent \textbf{Theorem~\ref{algebraic-dual-reed-muller}.}\textit{ 
Let $K$ be a field of characteristic $p$ and $\mathcal{X}$ a degenerate affine space as above.
Assume that $\gcd(p,e_i)=p$, where $e_i=|B_i|$,
for $i=1,\ldots,s$.  If $-1\leq d\leq r_0=\sum_{i=1}^s(e_i-1)$
and $\mathcal{L}=
K(S_{\leq d}\bigcap \Delta_\prec(I))$, then 
$$\mathcal{L}^\perp=K(\Delta_\prec(I)\setminus\{t^{b_1},\ldots,t^{b_k}\})=K(S_{\leq
r_0-d-1}\textstyle\bigcap\Delta_\prec(I)).
$$
}
\quad The codes $C_\mathcal{X}(d)^\perp$ and $C_\mathcal{X}(r_0-d-1)$
are monomially equivalent
because $I$ is a complete intersection
(Corollary~\ref{gorenstein-codes}). 
We show they are equal if ${\rm char}(K)$ divides $e_i$ for all $i$
(Proposition~\ref{dual-affine-degenerate-reed-muller}). 
When $\mathcal{X}=K^s$ the equality 
$C_\mathcal{X}(d)=C_\mathcal{X}(r_0-d-1)^\perp$, $r_0=s(q-1)$, has long been known; see for 
example \cite[Theorem~2.2.1]{delsarte-goethals-macwilliams} 
and \cite[Remark~4.7]{Pellikaan}.

\quad We include one section with examples (Section~\ref{examples-section}) and an appendix with
implementations of the algorithms in \textit{Macaulay}$2$ \cite{mac2}
that we used in the examples to compute bases for algebraic
duals, v-numbers, and standard indicator functions (Appendix~\ref{Appendix}).  

For all unexplained
terminology and additional information we refer the reader to 
\cite{CLO,cocoa-book,Sta1,monalg-rev} (for the theory of Gr\"obner bases and Hilbert
functions), and
\cite{Huffman-Pless,MacWilliams-Sloane,tsfasman} (for the theory of
error-correcting codes and linear codes).

\section{Preliminaries: Hilbert functions and vanishing ideals}\label{prelim-section} 
In this section we introduce Hilbert functions and characterize the
symmetry of the $h$-vector of the homogenization of a vanishing ideal.

Let $S=K[t_1,\ldots,t_s]=\bigoplus_{d=0}^\infty S_d$ be a polynomial ring
over a finite field $K=\mathbb{F}_q$ with the standard grading and let $I$ be an ideal of $S$. 
The Krull dimension of $S/I$ is
denoted by $\dim(S/I)$. We say that $I$ has \textit{dimension} $k$ if
$\dim(S/I)$ is equal to $k$. The \textit{height} of $I$, 
denoted ${\rm ht}(I)$, is $s-\dim(S/I)$. We set $S_{\leq
d}=\bigoplus_{i=0}^dS_i$, $S_{\leq -1}=\{0\}$, and $I_{\leq d}=I\textstyle\bigcap S_{\leq d}$. The function 
$$
H_I^a(d):=\dim_K(S_{\leq d}/I_{\leq d}),\ \ \ d=-1,0,1,2,\ldots,
$$
is  called the \textit{affine Hilbert function} of $S/I$. In
particular, $H_I^a(-1)=0$. For
simplicity we also call $H_I^a$ the affine Hilbert function of $I$.
The \textit{Hilbert function} of a graded ideal $J$ of $S$, denoted
$H_J$, is the function given by $H_J(d):=\dim_K(S_d/J_d)$ for $d\geq -1$, where
$J_d=S_d\textstyle\bigcap J$. 

Let $u=t_{s+1}$ be a new variable. For $f\in S$ of degree $e$ define 
$$
f^h:=u^ef\left({t_1}/{u},\ldots,{t_s}/{u}\right),
$$
that is,  $f^h$ is the homogenization of the polynomial 
$f$ with respect to $u$. The {\it homogenization\/}
of $I$ is the ideal $I^h$ of $S[u]$ given by 
$I^h:=(\{f^h|\, f\in I\})$, where $S[u]$ is given the 
standard grading. One has the following two well-known facts
\begin{equation}\label{nov13-20}
\dim(S[u]/I^h)=\dim(S/I)+1\mbox{ and } H_I^a(d)=H_{I^h}(d)
\mbox{ for }d\geq -1,
\end{equation}
where $H_{I^h}$ is the Hilbert function of the graded ideal $I^h$, 
see for instance
\cite[Lemma~8.5.4]{monalg-rev}. If $k=\dim(S/I)$, by a Hilbert
theorem \cite[p.~58]{Sta1}, there is a unique polynomial
$h^a_I(z)=\sum_{i=0}^{k}a_iz^i$ of degree $k$ in $\mathbb{Q}[z]$ such that
$h^a_I(d)=H_I^a(d)$ for $d\gg 0$. By convention the
degree of the zero polynomial is $-1$.   
The integer $k!\, a_k$, denoted ${\rm deg}(S/I)$, is called the
\textit{degree} of $S/I$.  The degree of $S/I$ is equal to 
$\deg(S[u]/I^{h})$. 
If $k=0$, then $H_I^a(d)=\deg(S/I)=\dim_K(S/I)$ for $d\gg 0$. Note
that the degree of $S/I$ is positive if $I\subsetneq S$ and is $0$
otherwise. 

We say that $I$ is a \textit{complete intersection} if $I$ can be
generated by ${\rm ht}(I)$ elements. The ideal $I$ and the ring $S/I$
are called \textit{Gorenstein} if the localization of $S/I$ at every
maximal ideal is a Gorenstein local ring in the sense of
\cite[Definition~2.8.3]{monalg-rev}. If $I=I(X)$ is the vanishing ideal of
a set of points in $K^s$, using Lemma~\ref{primdec-ixx} below,
it follows that $I$ is Gorenstein. Permitting an abuse of
terminology, we say that $I=I(X)$ is
\textit{Gorenstein} if $S[u]/I^h$ is a Gorenstein graded ring, that is, 
$S[u]/I^h$ is Cohen--Macaulay and the last Betti
number in the minimal graded resolution of $S[u]/I^h$ 
is equal to $1$ \cite[Corollary 5.3.5]{monalg-rev} (Example~\ref{Ivan-Hiram}). 

An element $f\in S$ is called a {\it zero-divisor\/} of $S/I$---as an
$S$-module---if there is
$\overline{0}\neq \overline{a}\in S/I$ such that
$f\overline{a}=\overline{0}$, and $f$ is called {\it regular\/} on
$S/I$ otherwise. Note that $f$ is a zero-divisor of $S/I$ if
and only if $(I\colon f)\neq I$. An associated prime of $I$ is a prime
ideal $\mathfrak{p}$ of $S$ of the form $\mathfrak{p}=(I\colon f)$
for some $f$ in $S$. The radical of $I$ is denoted by ${\rm rad}(I)$.
The ideal $I$ is \textit{radical} if $I={\rm rad}(I)$.

\begin{theorem}{\cite[Lemma~2.1.19,
Corollary~2.1.30]{monalg-rev}}\label{zero-divisors} If $I$ is an
ideal of $S$ and
$I=\mathfrak{q}_1\bigcap\cdots\bigcap\mathfrak{q}_m$ is
an irredundant primary decomposition with ${\rm
rad}(\mathfrak{q}_i)=\mathfrak{p}_i$, then the set of zero-divisors
$\mathcal{Z}_S(S/I)$  of $S/I$ is equal to
$\bigcup_{i=1}^m\mathfrak{p}_i$,
and $\mathfrak{p}_1,\ldots,\mathfrak{p}_m$ are the associated primes of
$I$.
\end{theorem}

Recall $C_{X}(d)$ denotes the Reed--Muller-type code of degree $d$ on
a set $X$ 
of points in $K^s$ and $\delta(C_{X}(d))$ represents the minimum distance of the code.

\begin{proposition}\cite[Corollary~2.6]{affine-codes}\label{behavior-hilbert-function}
Let $X$ be a subset of $K^s$ and let $I=I(X)$ be its
vanishing ideal. 
Then, $H_I^a$ is increasing until
it reaches the constant value $|X|$, and $\delta(C_{X}(d))$ is
decreasing, as a function of $d$, until 
it reaches the constant value $1$. In particular, $\deg(S/I)=|X|$.
\end{proposition}

If $I=I(X)$ and $X\subset K^s$, the least integer $r_0\geq 0$ such that 
$H_I^a(d)=|X|$ (resp. $H_{I^h}(d)=|X|$) for $d\geq r_0$, denoted
${\rm reg}(H_I^a)$ (resp. ${\rm reg}(H_{I^h})$), is called the
\textit{regularity index} of $H_I^a$ (resp. $H_{I^h}$). By
Eq.~\eqref{nov13-20}, $r_0={\rm reg}(H_I^a)={\rm
reg}(H_{I^h})$. It is known that ${\rm reg}(H_{I^h})$ equals the \textit{Castelnuovo--Mumford
regularity} of $S[u]/I^h$ in the sense of
\cite[p.~55]{eisenbud-syzygies}, see for instance \cite[p.~346]{monalg-rev}. For this reason ${\rm reg}(H_{I^h})$
is simply called the \textit{regularity} of $S[u]/I^h$.

\begin{lemma}\cite[Proposition~3.4.5]{monalg-rev}\label{oct3-20} Let
$X$ be a subset of $K^s$. Then, the ideal $I(X)^h$ is the homogeneous
vanishing ideal $I(Y)$ of the set
$Y:=\{[x,1] \mid x\in X\}$ of projective points in $\mathbb{P}^s$.
\end{lemma}

\begin{lemma}{\cite[p.~389]{cocoa-book}}\label{primdec-ixx} 
Let $X$ be a subset of
$K^s$, let $P$ be a point in $X$, $P=(p_1,\ldots,p_s)$, and
let $I_{P}$ be the vanishing ideal 
of $P$. Then $I_P$ is a maximal ideal of $S$ of height $s$, 
\begin{equation*}
I_P=(t_1-p_1,\ldots,t_s-p_s),\ \deg(S/I_P)=1, 
\end{equation*}
and $I(X)=\bigcap_{P\in X}I_{P}$ is the primary
decomposition of $I(X)$.  
\end{lemma}

\begin{lemma}\label{u-is-regular} 
Let $X$ be a subset of $K^s$. Then, the variable $u$ is regular on
$S[u]/I(X)^h$.
\end{lemma}

\begin{proof} We set $I=I(X)$. From Lemma~\ref{oct3-20}, we get
$I^h=\bigcap_{P\in X}I_{[P,1]}$. If $P=(p_1,\ldots,p_s)$ is a point in
$X$, then $I_{[P,1]}$ is generated by
$\mathcal{G}=\{t_1-p_1u,\ldots,t_s-p_su\}$.
Hence, by Theorem~\ref{zero-divisors}, it suffices to show that $u$ is not in
$I_{[P,1]}$. Pick a graded order with $t_1\succ\cdots\succ t_s\succ
u$. The set $\mathcal{G}$ is a Gr\"obner basis for $I_{[P,1]}$. If
$u$ is in $I_{[P,1]}$, 
then $u\in{\rm in}_\prec(I_{[P,1]})=(t_1,\ldots,t_s)$, a contradiction.
\end{proof}

Let $I\subset S$ be an ideal, let $\prec$ be a monomial order, and let
$\Delta_\prec(I)$ be the set of standard monomials of $S/I$. 
The image of $\Delta_\prec(I)$, under the canonical 
map $S\mapsto S/I$, $x\mapsto \overline{x}$, is a basis of $S/I$ as a
$K$-vector space \cite[Proposition~6.52]{Becker-Weispfenning}.

\begin{lemma}\label{lemma-referee1} 
Let $I\subset S$ be an ideal and let $\prec$ be a graded
monomial order on $S$. Then $H_I^a(d)$ is equal to
$H_{{\rm in}_\prec(I)}^a(d)$ for $d\geq 0$, $H_I^a(d)$ is
$|\Delta_\prec(I)\bigcap S_{\leq d}|$, the number of standard
monomials of $S/I$ of degree at most $d$, and $\dim(S/I)=\dim(S/{\rm
in}_\prec(I))$.  
\end{lemma}
\begin{proof} By \cite[Chapter 9,
Section 3, Propositions 3 and 4]{CLO}, we have that $H_{{\rm
in}_\prec(I)}^a(d)$ is the number of monomials 
of $S$ not in the ideal ${\rm in}_\prec(I)$ of degree $\leq d$, and $H_I^a(d)$ is equal to
$H_{{\rm in}_\prec(I)}^a(d)$ when $\prec$ is graded. Hence, $H_I^a(d)$ is the
number of standard monomials of $S/I$ of degree at most $d$. As $S/I$
and $S/{\rm in}_\prec(I)$ have the same affine Hilbert function, they 
have the same dimension.
\end{proof}

\begin{lemma}\label{sep11-20} Let $X$ be a subset of $K^s$, $I=I(X)$,
$r_0={\rm reg}(H_I^a)$, 
and let $\prec$ be a graded monomial order
on $S$. Then $\Delta_\prec(I)\subset S_{\leq r_0}$,
$\Delta_\prec(I)\not\subset S_{\leq r_0-1}$, and
$|X|=H_I^a(r_0)=|\Delta_\prec(I)|$.
\end{lemma}

\begin{proof} By Proposition~\ref{behavior-hilbert-function},
$H_I^a(r_0-1)<H_I^a(r_0)=|X|$. Hence, by Lemma~\ref{lemma-referee1}, it suffices to show
the inclusion $\Delta_\prec(I)\subset S_{\leq r_0}$. 
We proceed by contradiction assuming that 
$\Delta_\prec(I)\not\subset S_{\leq r_0}$. Pick a monomial 
$t^a$ in $\Delta_\prec(I)$ with $\deg(t^a)=d_0>r_0$. Then, one has
the strict inclusion $\Delta_\prec(I)\bigcap S_{\leq
r_0}\subsetneq\Delta_\prec(I)\bigcap
S_{\leq d_0}$ and, by Lemma~\ref{lemma-referee1}, one has
$|X|=H_I^a(r_0)<H_I^a(d_0)$, a contradiction because 
$H_I^a(d)=|X|$ for $d\geq r_0$ (Proposition~\ref{behavior-hilbert-function}).
\end{proof}

Let $J$ be a graded ideal of $S$ and let $F_J(z):=\sum_{i=0}^\infty
H_J(i)z^i$ be its Hilbert series. We now introduce the notion of
$h$-vector of $S/J$. By the Hilbert--Serre theorem
\cite{Sta1,monalg-rev} there is a (unique) polynomial
$h(z)=\sum_{i=0}^rh_iz^i$, $h_r\neq 0$, with integral 
coefficients such that $h(1)\neq 0$ and    
\begin{displaymath}
F_J(z)=\frac{h(z)}{(1-z)^{k}},
\end{displaymath}
where $k={\rm dim}(S/J)$. The $h$-{\it vector}
of $S/J$ is defined as $h(S/J):=(h_0,\ldots,h_r)$. We say that the 
$h$-vector of $S/J$ is \textit{symmetric} if $h_i=h_{r-i}$ for 
$0\leq i\leq r$. The $h$-vector of a Gorenstein graded algebra is 
symmetric \cite{Sta1}. For almost Gorenstein algebras
and coordinate rings of CB-schemes their $h$-vectors satisfy certain
interesting linear inequalities \cite{geramita-cayley-bacharach,Higashitani}.

\begin{proposition}\label{duality-hilbert-function} Let $I=I(X)$ be
the vanishing ideal of a subset $X$ of $K^s$, let $r_0$
be the regularity index of $H_I^a$, let $\prec$ be a
graded monomial order on $S[u]$ with $t_1\succ\cdots\succ
t_s\succ u$, and let $I^h$ be 
the homogenization of $I$ with respect to $u$. 
The following are equivalent.
\begin{enumerate}
\item[(a)] The $h$-vector of $S[u]/I^h$ is symmetric.
\item[(b)] $H_I^a(d)+H_I^a(r_0-d-1)=|X|$ for $-1\leq d\leq r_0$.
\item[(c)] $H_{{\rm in}_\prec(I)}(d)=H_{{\rm in}_\prec(I)}(r_0-d)$ for
$0\leq d\leq r_0$.
\item[(d)] $|\Delta_\prec(I)\bigcap S_d|=|\Delta_\prec(I)\bigcap S_{r_0-d}|$
for $0\leq d\leq r_0$.
\end{enumerate}
\end{proposition}

\begin{proof} As $S[u]/I^h$ is Cohen--Macaulay of dimension $1$,
its Hilbert series can be written as
\begin{equation}\label{aug17-20}
F_{I^h}(z)=\frac{h_0+h_1z+\cdots+h_{r_0}z^{r_0}}{1-z},
\end{equation}
where $h(z)=h_0+h_1z+\cdots+h_{r_0}z^{r_0}$ is a polynomial
with positive integer coefficients and the degree and regularity of
$S[u]/I^h$ are $h(1)$ and $r_0$, respectively \cite{Sta1,monalg-rev}.
The ideal $I$ (resp. $I^h$) and its initial
ideal ${\rm in}_\prec(I)$ (resp. ${\rm in}_\prec(I^h)$) have the same
affine Hilbert function (resp. Hilbert function)
(Lemma~\ref{lemma-referee1}). As $u$ is not in the ideal ${\rm
in}_\prec(I^h)$, there is an exact sequence 
$$
0\longrightarrow(S[u]/{\rm
in}_\prec(I^h))[-1]\stackrel{u}{\longrightarrow}S[u]/{\rm
in}_\prec(I^h){\longrightarrow}S[u]/({\rm in}_\prec(I^h),u){\longrightarrow}
0.
$$
\quad Hence, noticing the equalities 
${\rm in}_\prec(I^h)={\rm in}_\prec(I)S[u]$ and 
$S[u]/({\rm in}_\prec(I^h),u)=S/{\rm in}_\prec(I)$ and, by taking
Hilbert series in this exact sequence, we obtain
$$
F_{I^h}(z)=zF_{I^h}(z)+H_{{\rm in}_\prec(I)}(0)+H_{{\rm
in}_\prec(I)}(1)z+\cdots+H_{{\rm in}_\prec(I)}(r_0)z^{r_0}. 
$$
\quad Therefore, by Eq.~\eqref{aug17-20}, the $h$-vectors of 
$S/{\rm in}_\prec(I)$ and $S[u]/I^h$ are equal and  
\begin{equation}\label{aug17-20-1}
h_i=H_{{\rm in}_\prec(I)}(i)\mbox{ for }0\leq i\leq r_0. 
\end{equation}
\quad (a) $\Rightarrow$ (b): Note that when $d=-1$ or $d=r_0$ (b) holds by the definition of $r_0$, so we
may assume that $0\leq d<r_0$. Now assume that $h(S[u]/I^h)=(h_0,\ldots,h_{r_0})$
is symmetric. Hence, by Eq.~\eqref{aug17-20-1}, 
we obtain $H_{{\rm in}_\prec(I)}(i)=H_{{\rm in}_\prec(I)}(r_0-i)$ for
$0\leq i\leq r_0$. The affine Hilbert function of $I$ in degree $d$ is given by 
$H_I^a(d)=H_{{\rm in}_\prec(I)}^a(d)=\sum_{i=0}^dH_{{\rm in}_\prec(I)}(i)$
(Lemma~\ref{lemma-referee1}). Therefore
\begin{eqnarray*}
|X|&=&\deg(S[u]/I^h)=\sum_{i=0}^{r_0}h_i=\sum_{i=0}^{r_0}H_{{\rm
in}_\prec(I)}(i)=\sum_{i=0}^{d}H_{{\rm in}_\prec(I)}(i) +\sum_{i=d+1}^{r_0}H_{{\rm
in}_\prec(I)}(i)\\
&=&H_I^a(d)+\sum_{i=0}^{r_0-d-1}H_{{\rm
in}_\prec(I)}(r_0-i)=H_I^a(d)+\sum_{i=0}^{r_0-d-1}H_{{\rm
in}_\prec(I)}(i)\\
&=&H_I^a(d)+H_I^a(r_0-d-1). 
\end{eqnarray*}
\quad (b) $\Rightarrow$ (c): As $|X|=H_I^a(d)+H_I^a(r_0-d-1)$ and
$|X|=H_I^a(d-1)+H_I^a(r_0-d)$, by adding the following two equalities
\begin{align*}
H_I^a(d)&=\sum_{i=0}^dH_{{\rm in}_\prec(I)}(i)=H_I^a(d-1)+H_{{\rm
in}_\prec(I)}(d),\\
H_I^a(r_0-d-1)&=\sum_{i=0}^{r_0-d-1}H_{{\rm
in}_\prec(I)}(i)=H_I^a(r_0-d)-H_{{\rm
in}_\prec(I)}(r_0-d),
\end{align*}
we obtain the equality $H_{{\rm in}_\prec(I)}(d)=H_{{\rm in}_\prec(I)}(r_0-d)$.

(c) $\Rightarrow$ (a): The symmetry of the $h$-vector of $S[u]/I^h$
follows from Eq.~\eqref{aug17-20-1}.

(c) $\Leftrightarrow$ (d): The number of standard monomials of $I$ of
degree $d$ is $H_{{\rm in}_\prec(I)}(d)$ \cite[p.~433]{CLO}, that is,
$H_{{\rm in}_\prec(I)}(d)$ is equal to $|\Delta_\prec(I)\bigcap S_d|$. 
Hence (c) and (d) are equivalent.
\end{proof}

\begin{corollary}\label{duality-hilbert-gorenstein} 
Let $I=I(X)$ be
the vanishing ideal of a subset $X$ of $K^s$, 
let $\prec$ be a graded monomial order on $S$, and let $r_0$ be the
regularity index of $H_I^a$. The following hold. 
\begin{enumerate}
\item[(a)] If $I$ is Gorenstein, then $H_I^a(d)+H_I^a(r_0-d-1)=\deg(S/I)=|X|$ 
for $-1\leq d\leq r_0$.
\item[(b)] If $I$ is Gorenstein, then there is only one standard
monomial of degree $r_0$. 
\item[(c)] If $I$ is generated by a Gr\"obner
basis $\mathcal{G}=\{g_1,\ldots,g_s\}$ of $s$ elements, then $I$ is Gorenstein.
\end{enumerate}
\end{corollary}

\begin{proof} (a): As $S[u]/I^h$ is a graded Gorenstein algebra of dimension
$1$, its $h$-vector is symmetric \cite[Theorems~4.1 and 4.2]{Sta1}.
Then, by Proposition~\ref{duality-hilbert-function}, the equality
follows.

(b): By part (a) and Proposition~\ref{duality-hilbert-function}(d),
one has $|\Delta_\prec(I)\bigcap S_d|=|\Delta_\prec(I)\bigcap S_{r_0-d}|$
for $0\leq d\leq r_0$. Setting $d=0$ in this equality, 
we get $1=|\Delta_\prec(I)\bigcap S_{r_0}|$.

(c): As $I$ is generated by the Gr\"obner basis $\mathcal{G}$, 
one has $I^h=(g_1^h,\ldots,g_s^h)$ \cite[p.~132]{monalg-rev}. The ideals $I$ and $I^h$ have
height $s$. Then, $I^h$ is a graded ideal of height $s$ generated by
$s$ homogeneous polynomials forming a regular sequence. Hence, by
\cite[Corollary~21.19]{Eisen}, $I^h$ is Gorenstein. 
\end{proof}

\section{The dual of evaluation codes}\label{dual-section}
To avoid repetitions, we continue to employ 
the notations and definitions used in Sections~\ref{intro-section} and
\ref{prelim-section}. In this section we show that the dual of an
evaluation code is the evaluation code of the algebraic dual. We 
give an effective criterion to determine whether or not 
the algebraic dual is monomial and show an algorithm that can be used
to compute a basis for the algebraic dual.

\begin{proposition}\cite{toric-codes}\label{transforming-new} 
Let $\mathcal{L}_X$ be an evaluation code on $X$, let $\prec$ be a
monomial order, let $\mathcal{G}$ be a Gr\"obner basis of
$I=I(X)$, let $\{h_1,\ldots,h_k\}$ be a subset of
$\mathcal{L}\setminus\{0\}$ and
for each $i$, let $r_i$ be the remainder on division of $h_i$ by 
$\mathcal{G}$. If $\mathcal{L}=K\{h_1,\ldots,h_k\}$ and 
$$
\widetilde{\mathcal{L}}:=K\{r_1,\ldots,r_k\},
$$
then $\widetilde{\mathcal{L}}\subset K\Delta_\prec(I)$,
$\widetilde{\mathcal{L}}_X$ is a standard evaluation code on $X$
relative to $\prec$ 
and $\widetilde{\mathcal{L}}_X=\mathcal{L}_X$.
\end{proposition}

\begin{corollary}\label{unique-standard}
Let $\mathcal{L}_X$ be an evaluation code on $X$ and let $\prec$ be a
monomial order. Then there exists a unique linear subspace
$\widetilde{\mathcal{L}}$ of $K\Delta_\prec(I)$ such 
that $\widetilde{\mathcal{L}}_X=\mathcal{L}_X$.
\end{corollary}

\begin{proof} The existence follows from
Proposition~\ref{transforming-new}. Assume that $\mathcal{L}_1$ and
$\mathcal{L}_2$ are two linear subspaces of $K\Delta_\prec(I)$ such 
that $(\mathcal{L}_1)_X=(\mathcal{L}_2)_X$. Let $P_1,\ldots,P_m$ be
the points of $X$. To show
the inclusion $\mathcal{L}_1\subset\mathcal{L}_2$ take
$f\in\mathcal{L}_1$. Then ${\rm ev}(f)=(f(P_1),\ldots,f(P_m))$ is in
$(\mathcal{L}_1)_X$. Thus there is $g\in\mathcal{L}_2$ such that
${\rm ev}(f)={\rm ev}(g)=(g(P_1),\ldots,g(P_m))$. Hence $f-g\in I(X)$
and $f-g=0$ because $f$ and $g$ are in $K\Delta_\prec(I)$. Thus
$f\in\mathcal{L}_2$. The inclusion $\mathcal{L}_2\subset\mathcal{L}_1$
follows from similar reasons. Therefore
$\mathcal{L}_1=\mathcal{L}_2$ and $\widetilde{\mathcal{L}}$ is unique.
\end{proof}

Recall that $\varphi$ is the $K$-linear map given by 
$$
\varphi\colon S\rightarrow K,\quad f\mapsto f(P_1)+\cdots+f(P_m).
$$

\begin{definition} Let $\mathcal{L}_X$ be an evaluation code on $X$
and let $\prec$ be a monomial order on $S$. The unique linear subspace
$\widetilde{\mathcal{L}}$ of $K\Delta_\prec(I)$ such 
that $\widetilde{\mathcal{L}}_X=\mathcal{L}_X$ is called
the \textit{standard function space} of $\mathcal{L}_X$. 
The \textit{dual} of
$\mathcal{L}$, denoted $\mathcal{L}^\perp$, is the $K$-linear space given by  
$\mathcal{L}^\perp:=({\rm ker}(\varphi)\colon
\mathcal{L})\bigcap K\Delta_\prec(I)$. We will also call $\mathcal{L}^\perp$ the \textit{algebraic dual} of 
$\mathcal{L}_X$ relative to $\prec$. 
\end{definition}

\begin{lemma}\label{dual-standard} Let $\mathcal{L}_X$ be an evaluation code on $X$ and 
let $\prec$ be a monomial order on $S$. If $I=I(X)$ and
$\widetilde{\mathcal{L}}$
is the standard function space of $\mathcal{L}_X$, then 
$$
\mathcal{L}^\perp=({\rm ker}(\varphi)\colon
\mathcal{L})\textstyle\bigcap K\Delta_\prec(I)=({\rm ker}(\varphi)\colon
\widetilde{\mathcal{L}})\textstyle\bigcap K\Delta_\prec(I)=\widetilde{\mathcal{L}}^\perp.
$$
\end{lemma}
\begin{proof} There are $g_1,\ldots,g_k$ in
$\mathcal{L}\setminus\{0\}$ such that
$\mathcal{L}=K\{g_1,\ldots,g_k\}$. By the division algorithm \cite[Theorem~3,
p.~63]{CLO}, for each
$i$, we can write $g_i=h_i+r_i$ for some $h_i\in I$ and $r_i\in
K\Delta_\prec(I)$. By Proposition~\ref{transforming-new}, one has 
$\widetilde{\mathcal{L}}=K\{r_1,\ldots,r_k\}$. To show the inclusion
``$\subset$'' take $f\in\mathcal{L}^\perp$, that is,
$f\mathcal{L}\subset\ker(\varphi)$ and $f\in K\Delta_\prec(I)$. Then 
$fr_i\in\ker(\varphi)$ for all $i$, and consequently
$f\in \widetilde{\mathcal{L}}^\perp$. To show the inclusion
``$\supset$'' take $f\in \widetilde{\mathcal{L}}^\perp$, that is, 
$f\widetilde{\mathcal{L}}\subset\ker(\varphi)$ and $f\in K\Delta_\prec(I)$. Then 
$fr_i\in\ker(\varphi)$ for all $i$ and, since $r_i=g_i-h_i$, we get
$fg_i\in\ker(\varphi)$ for all $i$. Thus, $f\in{\mathcal{L}}^\perp$.
\end{proof}

\begin{theorem}\label{formula-dual}
Let $\mathcal{L}_X$ be an evaluation code on $X$ and let $I=I(X)$ be the vanishing ideal
of $X$. If  $\prec$ is a monomial order and 
$\mathcal{L}^\perp=({\rm ker}(\varphi)\colon
\mathcal{L})\bigcap K\Delta_\prec(I)$, then $(\mathcal{L}_X)^\perp$ is the standard evaluation code
$(\mathcal{L}^\perp)_X$ on $X$ relative 
to $\prec$.
\end{theorem}

\begin{proof} First we show the inclusion
$(\mathcal{L}_X)^\perp\subset (\mathcal{L}^\perp)_X$. Take $\alpha\in
(\mathcal{L}_X)^\perp$. Let $r_0$ be the regularity index of $H_{I}^a$. The
evaluation map 
$${\rm ev}_{r_0}\colon S_{\leq r_0}\rightarrow K^{m},\quad
f\mapsto\left(f(P_1),\ldots,f(P_m)\right),
$$ 
is surjective since $H_I^a(r_0)=\dim_K(S_{\leq r_0}/I_{\leq
r_0})=|X|=m$. Then, $\alpha=(g_1(P_1),\ldots,g_1(P_m))$ for some 
$g_1\in S_{\leq r_0}$. By the division algorithm \cite[Theorem~3,
p.~63]{CLO}, we can write $g_1=g_2+g$, where $g_2\in
I$ and $g\in K\Delta_\prec(I)$. Thus, $\alpha=(g(P_1),\ldots,g(P_m))$.
Using that $\alpha\in (\mathcal{L}_X)^\perp$, we obtain 
$$
\langle\alpha,(f(P_1),\ldots,f(P_m))\rangle=\sum_{i=1}^mg(P_i)f(P_i)=
\sum_{i=1}^m(gf)(P_i)=0
$$
for all $f\in\mathcal{L}$. Thus, $g\in({\rm ker}(\varphi)\colon
\mathcal{L})\bigcap K\Delta_\prec(I)=\mathcal{L}^\perp$. From the
equality 
$$
(\mathcal{L}^\perp)_X=\{(h(P_1),\ldots,h(P_m)) \mid 
h \in \mathcal{L}^\perp\},
$$
we obtain $\alpha\in(\mathcal{L}^\perp)_X$. To show the inclusion 
$(\mathcal{L}_X)^\perp\supset (\mathcal{L}^\perp)_X$ take 
$\alpha\in (\mathcal{L}^\perp)_X$, that is,
$\alpha=(g(P_1),\ldots,g(P_m))$ for some $g\in\mathcal{L}^\perp$.
Then, $gf\in{\rm ker}(\varphi)$ for all $f\in\mathcal{L}$ and 
$$
\langle\alpha,(f(P_1),\ldots,f(P_m))\rangle=0
$$ 
for all $f\in\mathcal{L}$. From the equality
$\mathcal{L}_X=\{(f(P_1),\ldots,f(P_m)) \mid 
f \in \mathcal{L}\}$, we obtain $\alpha\in (\mathcal{L}_X)^\perp$. 
\end{proof}

We will need the following observation.

\begin{lemma}\cite[Lemma~3.1]{toric-codes}\label{apr26-20} 
Let $X$ be a subset of $K^s$ and let $\mathcal{L}_X$ be a standard evaluation code on $X$
relative to a monomial order $\prec$. Then, $\mathcal{L}\bigcap I(X)=(0)$ 
and $\mathcal{L}\simeq\mathcal{L}_X$.
\end{lemma}

\begin{proof} We set $I=I(X)$. Take $f\in\mathcal{L}\bigcap I$ and 
recall that $\mathcal{L}$ is a linear subspace of $K\Delta_\prec(I)$. If $f\neq 0$, 
then ${\rm in}_\prec(f)\in{\rm in}_\prec(I)$, a contradiction since
all monomials of $f$ are standard monomials of $S/I$. Thus, $f=0$. Hence, the
evaluation map gives an isomorphism between $\mathcal{L}$ and
$\mathcal{L}_X$. 
\end{proof}

The next result shows that the
dual of $\mathcal{L}$ behaves well.

\begin{proposition}\label{dual-properties} Let $\mathcal{L}_X$ be a standard evaluation code
on $X$ relative to a monomial order $\prec$ on $S$ and let $I=I(X)$. The following hold.
\begin{enumerate}
\item[\rm(a)] $\dim_K(\mathcal{L})+\dim_K(\mathcal{L}^\perp)=|X|$.
\item[\rm(b)] The conditions $(\mathrm{b}_1)$-$(\mathrm{b}_3)$ are
equivalent
$$
(\mathrm{b}_1)\ \mathcal{L}_X\textstyle\bigcap(\mathcal{L}_X)^\perp=(0),\quad 
(\mathrm{b}_2) \ \mathcal{L}\bigcap\mathcal{L}^\perp=(0),\quad
(\mathrm{b}_3) \ \mathcal{L}+\mathcal{L}^\perp=K\Delta_\prec(I).
$$ 
\item[\rm(c)] $\mathcal{L}_X=(\mathcal{L}_X)^\perp$ if and only if
$\mathcal{L}=\mathcal{L}^\perp$.
\item[\rm(d)] $(\mathcal{L}^\perp)^\perp=\mathcal{L}$.
\end{enumerate}
\end{proposition}

\begin{proof} (a): By \cite[Theorem~1.2.1]{Huffman-Pless},
Theorem~\ref{formula-dual} and Lemma~\ref{apr26-20}, we get 
\begin{align*}
|X|&=\dim_K(\mathcal{L}_X)+\dim_K(\mathcal{L}_X)^\perp=
\dim_K(\mathcal{L}_X)+\dim_K(\mathcal{L}^\perp)_X\\
&=\dim_K(\mathcal{L})+\dim_K(\mathcal{L}^\perp).
\end{align*}
\quad (b): $(\mathrm{b}_1)\Rightarrow(\mathrm{b}_2)$ Assume
$\mathcal{L}_X\bigcap(\mathcal{L}_X)^\perp=(0)$ and take
$g\in\mathcal{L}\bigcap\mathcal{L}^\perp$. Then, ${\rm
ev}(g)\in\mathcal{L}_X\bigcap(\mathcal{L}^\perp)_X$ and, because of
Theorem~\ref{formula-dual}, we get that ${\rm ev}(g)$ is in
$\mathcal{L}_X\bigcap(\mathcal{L}_X)^\perp=(0)$ and ${\rm ev}(g)=0$.
Hence, $g\in I$, and consequently $g=0$ 
because $g\in K\Delta_\prec(I)$.

$(\mathrm{b}_2)\Rightarrow(\mathrm{b}_3)$ Assume
$\mathcal{L}\bigcap\mathcal{L}^\perp=(0)$. By part (a) one has
\begin{equation*}
|X|=\dim_K(\mathcal{L})+\dim_K(\mathcal{L}^\perp)=\dim_K(\mathcal{L}
+\mathcal{L}^\perp)+\dim_K(\mathcal{L}\textstyle\bigcap\mathcal{L}^\perp).
\end{equation*}
\quad Hence, $|X|=\dim_K(\mathcal{L}+\mathcal{L}^\perp)$. From the
inclusion $\mathcal{L}+\mathcal{L}^\perp\subset K\Delta_\prec(I)$ and
noticing that these linear spaces have dimension $|X|$ (Lemma~\ref{sep11-20}), we get
$\mathcal{L}+\mathcal{L}^\perp=K\Delta_\prec(I)$.

$(\mathrm{b}_3)\Rightarrow(\mathrm{b}_1)$ Assume
$\mathcal{L}+\mathcal{L}^\perp=K\Delta_\prec(I)$. The evaluation map
``${\rm ev}$'' induces an isomorphism between $K\Delta_\prec(I)$ and
$K^{|X|}$. Then, by Theorem~\ref{formula-dual}, we get
$$\mathcal{L}_X+(\mathcal{L}^\perp)_X=\mathcal{L}_X+(\mathcal{L}_X)^\perp=K^{|X|}$$
and the dimension of $\mathcal{L}_X+(\mathcal{L}_X)^\perp$ is $|X|$.
Therefore, from the equality
\begin{equation*}
|X|=\dim_K(\mathcal{L}_X)+\dim_K(\mathcal{L}_X)^\perp=\dim_K(\mathcal{L}_X
+(\mathcal{L}_X)^\perp)+\dim_K(\mathcal{L}_X\textstyle\bigcap(\mathcal{L}_X)^\perp),
\end{equation*}
we obtain $\mathcal{L}_X\bigcap(\mathcal{L}_X)^\perp=(0)$.

(c): $\Rightarrow$) Assume $\mathcal{L}_X=(\mathcal{L}_X)^\perp$. Let
$P_1,\ldots,P_m$ be the points of $X$. First we show the inclusion
$\mathcal{L}\subset\mathcal{L}^\perp$. Take $f\in\mathcal{L}$. Then,  
${\rm ev}(f)=(f(P_1),\ldots,f(P_m))$ is in $\mathcal{L}_X$. By 
Theorem~\ref{formula-dual}, $(\mathcal{L}_X)^\perp$ is equal to
$(\mathcal{L}^\perp)_X$. Thus, there is $g\in\mathcal{L}^\perp$ such
that ${\rm ev}(f)={\rm ev}(g)=(g(P_1),\ldots,g(P_m))$. Then $f-g\in
I$, and $f=g$ because $f,g$ are in $K\Delta_\prec(I)$. Thus,
$f\in\mathcal{L}^\perp$. Now we show the inclusion $\mathcal{L}^\perp\subset\mathcal{L}$. Take
$f\in\mathcal{L}^\perp$. Then, ${\rm ev}(f)=(f(P_1),\ldots,f(P_m))$
is in $(\mathcal{L}^\perp)_X$. By 
Theorem~\ref{formula-dual}, $(\mathcal{L}^\perp)_X$ is equal to
$(\mathcal{L}_X)^\perp=\mathcal{L}_X$. Thus, there is
$g\in\mathcal{L}$ such
that ${\rm ev}(f)={\rm ev}(g)=(g(P_1),\ldots,g(P_m))$. Then $f-g\in
I$, and $f=g$ because $f,g$ are standard polynomials. Thus,
$f\in\mathcal{L}$. 

$\Leftarrow$) Assume $\mathcal{L}=\mathcal{L}^\perp$. Then, by
Theorem~\ref{formula-dual}, $\mathcal{L}_X=(\mathcal{L}^\perp)_X=(\mathcal{L}_X)^\perp$.

(d): To show the inclusion
$(\mathcal{L}^\perp)^\perp\subset\mathcal{L}$ take
$g\in(\mathcal{L}^\perp)^\perp$. Then $gf\in\ker(\varphi)$ for all
$f\in\mathcal{L}^\perp$. Hence $\langle{\rm ev}(g),{\rm ev}(f)
\rangle=0$ for all $f\in\mathcal{L}^\perp$, that is, ${\rm
ev}(g)\in((\mathcal{L}^\perp)_X)^\perp$. By
Theorem~\ref{formula-dual} and the fact that the dual of
$(\mathcal{L}_X)^\perp$ is equal to $\mathcal{L}_X$ \cite[p.~26]{MacWilliams-Sloane}, one has
$((\mathcal{L}^\perp)_X)^\perp=((\mathcal{L}_X)^\perp)^\perp=\mathcal{L}_X$.
Thus, ${\rm ev}(g)\in\mathcal{L}_X$ and there is $h\in\mathcal{L}$
such that ${\rm ev}(g)={\rm ev}(h)$. From this equality we get that
$g-h$ is in $I$. As $g,h$ are in $K\Delta_\prec(I)$, it follows that
$g=h$ and $g\in\mathcal{L}$. To show the inclusion $\mathcal{L}\subset(\mathcal{L}^\perp)^\perp$
take $f\in\mathcal{L}$. Then ${\rm ev}(f)\in\mathcal{L}_X$. For any
$g\in\mathcal{L}^\perp$, one has $g\mathcal{L}\subset\ker(\varphi)$.
In particular, $gf\in\ker(\varphi)$ for any $g\in\mathcal{L}^\perp$,
and consequently $f$ is in $(\ker(\varphi)\colon\mathcal{L}^\perp)\bigcap
K\Delta_\prec(I)=(\mathcal{L}^\perp)^\perp$.
\end{proof}

\begin{proposition}\label{binomial-standard} Let $\mathcal{L}_X$ be an evaluation code 
on $X$ and let $\prec$ be a monomial order on $S$. 
If $I(X)$ is a binomial ideal and $\mathcal{L}$ is a monomial space,
then $\mathcal{L}_X$ is a standard monomial code. 
\end{proposition}

\begin{proof} There exists a Gr\"obner basis
$\mathcal{G}$ of $I(X)$ consisting of binomials \cite[Lemma
8.2.17]{monalg-rev}. The linear space $\mathcal{L}$ is generated by a
finite set $\{t^{a_1},\ldots,t^{a_k}\}$ of monomials. By the division
algorithm \cite[Theorem~3, p.~63]{CLO} it follows that the remainder
$r_i$ on division of $t^{a_i}$ by $\mathcal{G}$ is a monomial. Hence,
$\widetilde{\mathcal{L}}=K\{r_1,\ldots,r_k\}$ is a monomial space. 
\end{proof}

Let $A=\{x^{c_1},\ldots,x^{c_s}\}$ be a finite set of monomials in a
polynomial ring $K[x_1,\ldots,x_n]$. 
The {\it affine set parameterized\/} by $A$ is the set $X$ of all 
points $(x^{c_1}(\alpha),\ldots,x^{c_s}(\alpha))$ such that $\alpha\in
K^n$.
The next result gives a wide class of standard monomial codes 
that includes the family of parameterized affine codes \cite{affine-codes} and the subfamily
of $q$-ary Reed--Muller codes \cite{Pellikaan}. 

\begin{corollary}\label{parameterized-binomial} If $X$ is
parameterized by monomials and $\mathcal{L}$ is a monomial space, then $\mathcal{L}_X$ is a standard
monomial code. In
particular if $\mathcal{L}=S_{\leq d}$, then $\mathcal{L}_X$ is a
standard monomial code.  
\end{corollary}

\begin{proof} By \cite[Theorem~4, p. 435]{vanishing-ideals}, $I(X)$ is a binomial
ideal. Hence, by Proposition~\ref{binomial-standard}, $\mathcal{L}_X$
is a standard monomial code.
\end{proof}

Using the equality $(\mathcal{L}_X)^\perp=(\mathcal{L}^\perp)_X$
(Theorem~\ref{formula-dual}) and the next result we obtain an
effective 
criterion to verify whether or not 
the dual of an evaluation code is a standard monomial code. 

\begin{proposition}\label{monomial-criterion} Let $\mathcal{L}_X$ be
an evaluation code on $X$, let $I$ be the vanishing ideal of
$X$, and let $\prec$ be a monomial order. Then, 
$(\mathcal{L}^\perp)_X$ is a standard monomial code on $X$ if and
only if 
$$|({\rm ker}(\varphi)\colon
\mathcal{L})\textstyle\bigcap\Delta_\prec(I)|=|X|-\dim_K(\mathcal{L}_X).
$$
\end{proposition}

\begin{proof} By Corollary~\ref{unique-standard}, the standard function space of
$(\mathcal{L}^\perp)_X$ is equal to $\mathcal{L}^\perp$ because
$\mathcal{L}^\perp$ is 
generated by standard polynomials of $S/I$. Then, as
$(\mathcal{L}^\perp)_X$ is a standard evaluation code, one has
$\mathcal{L}^\perp\simeq(\mathcal{L}^\perp)_X$ 
(Lemma~\ref{apr26-20}). Hence, by Theorem~\ref{formula-dual}, we get 
\begin{equation}\label{jul3-20}
\dim_K(\mathcal{L}^\perp)=\dim_K(\mathcal{L}^\perp)_X=
\dim_K(\mathcal{L}_X)^\perp=|X|-\dim_K(\mathcal{L}_X).
\end{equation}
\quad $\Rightarrow$) Let ${B}$ be a finite monomial $K$-basis for
$\mathcal{L}^\perp$. By Eq.~\eqref{jul3-20}, $|{B}|$ is equal
to $|X|-\dim_K(\mathcal{L}_X)$. Hence, the desired equality follows by noticing 
that $({\rm ker}(\varphi)\colon\mathcal{L})\bigcap\Delta_\prec(I)={B}$. 

$\Leftarrow$) There are monomials
$t^{a_1},\ldots,t^{a_n}$ in $({\rm ker}(\varphi)\colon
\mathcal{L})\bigcap\Delta_\prec(I)$ with
$n=|X|-\dim_K(\mathcal{L}_X)$. By Eq.~(\ref{jul3-20}), one has
$\dim(\mathcal{L}^\perp)=n$. Hence, as $t^{a_i}\in\mathcal{L}^\perp$
for all $i$, we get $\mathcal{L}^\perp=K\{t^{a_1},\ldots,t^{a_n}\}$. 
\end{proof}

\subsection{Computing a basis} In this subsection we show an
effective algorithm to compute the dimension and a $K$-basis for a linear subspace of
$S$ of finite dimension. Let $(S^*)^{<\omega}$ be the set of finite
subsets of $S^*=S\setminus\{0\}$, let $\prec$ be the graded reverse
lexicographical order (GRevLex order) on
$S$, and let $\sigma$ and $\phi$ be the
functions
\begin{align*}
&\sigma,\, \phi\colon (S^*)^{<\omega}\rightarrow(S^*)^{<\omega},\quad
\sigma(A)=\{g\in A \mid {\rm in}_\prec(g)={\rm  
in}_\prec(\max(A))\},\\
&\phi(A)= \left(\{\max(A)-({\rm lc}(\max(A))/{\rm lc}(g))g \mid
g\in\sigma(A)\}\setminus\{0\}\right)\textstyle\bigcup(A\setminus\sigma(A)),
\end{align*}
where ${\rm lc}(g)$ denotes the leading coefficient of $g$ and 
$\max(A)$ is any polynomial in $A$ whose initial monomial is
$\max\{{\rm in}_\prec(g) : g\in A\}$. Note that $\sigma(A)$ is the set
of all polynomials in $A$ with largest initial monomial relative to
$\prec$ and, hence, is independent of the choice of $\max(A)$. 
The following result is based on Gaussian elimination.

\begin{theorem}{\rm(Basis algorithm)}\label{dim-algo}
\rm Let $\mathcal{L}=KA$ be a subspace of $S$ generated by a finite
subset $A$ of $S^*$. Then one can construct 
a $K$-basis for $\mathcal{L}$ using the following 
algorithm:
\begin{tabbing}
\ \ Input: $A$ \\ 
\ \ Output: a $K$-basis $B$ for $\mathcal{L}$\\ 
\ \ Initialization: $B:=A$ \\
\ \ while $B\neq\emptyset$ list $\max(B)$ do $B:=\phi(B)$.
\end{tabbing}
\end{theorem}

\begin{proof} As $\sigma(B)\subset B$ we can write 
$B=\{g_1,\ldots,g_n\}$, where $\sigma(B)=\{g_1,\ldots,g_r\}$,
$r\leq n$, and ${\rm in}_\prec(g_1)$ is equal to ${\rm in}_\prec(g_i)$ for
$i=1,\ldots,r$. Any $g_i\in\sigma(B)$ can be chosen to be $\max(B)$.
Setting $g_1:=\max(B)$ and 
$h_i:=g_1-({\rm lc}(g_1)/{\rm lc}(g_i))g_i$ for $i=1,\ldots,r$, one
has
\begin{equation*}
\phi(B)=
\left(\{h_i\}_{i=1}^r\setminus\{0\}\right)\textstyle\bigcup\{g_i\}_{i=r+1}^n.
\end{equation*}
\quad Note that ${\rm in}_\prec(g_1)\succ{\rm in}_\prec(g_i)$ for $i>r$
and ${\rm in}_\prec(g_1)\succ{\rm in}_\prec(h_i)$ for $i=2,\ldots,r$.
Thus, $\max(B)\succ\max(\phi(B))$ and the algorithm terminates after a
finite number of steps. If the algorithm terminates at $B$, that is,
$B\neq\emptyset$ and $\phi(B)=\emptyset$, then $B=\sigma(B)$, 
$g_1=({\rm lc}(g_1)/{\rm lc}(g_i))g_i$ for $i=1,\ldots,r$, and
$KB=Kg_1=K\max(B)$. That the output is a generating set for
$\mathcal{L}=KA$ follows by noticing that $KB=K\max(B)+K\phi(B)$.
Finally, we show that the output is linearly independent over $K$. The
output is the list 
$$
{B}=\{\max(A),\max(\phi(A)),\max(\phi(\phi(A))),\ldots,\max(\phi^{k-1}(A))\},
$$
where $\phi^{k-1}(A)\neq\emptyset$ and $\phi^k(A)=\emptyset$. Since
$\max(\phi^{i-1}(A))\succ\max(\phi^{i}(A))$ for $i=1,\ldots k-1$ it 
is not hard to see that ${B}$ is linearly independent.  
\end{proof}

\section{Indicator functions and v-numbers of vanishing ideals}\label{v-number-section}

Recall that  $K$ is a finite field, $X=\{P_1,\ldots,P_m\}$ is 
a set of points in $K^s$, $|X|\geq 2$, and $I=I(X)$ is its vanishing ideal.
Fix a graded monomial order $\prec$ on $S$. In this section we introduce
and study the v-number of $I$ and the indicator functions of $X$.

We begin with the notion of an indicator function
of a point in $X$ \cite{sorensen}. For a projective point an indicator
function is called a separator \cite[Definition~2.1]{geramita-cayley-bacharach}.  

\begin{definition} Let $X=\{P_1,\ldots,P_m\}$ be a subset of
$K^s$. A polynomial $f\in S$ is called an
\textit{indicator function} for $P_i$ if $f(P_i)\neq 0$ and $f(P_j)=0$ for all
$j\neq i$.
\end{definition}

An indicator function $f$ for $P_i$ can be normalized to have value 
$1$ at $P_i$ by considering $f/f(P_i)$. The following lemma lists basic properties of indicator functions.

\begin{lemma}\label{if} {\rm (a)} If $f,g$ are indicator functions for
$P_i$ in $K\Delta_\prec(I)$, then $g(P_i)f=f(P_i)g$.

{\rm (b)} The set of indicator functions for $P_i$ is
$(I\colon\mathfrak{p}_i)\setminus I$, where $\mathfrak{p}_i$ is the
vanishing ideal of $P_i$.

{\rm (c)} There exists a unique, up to multiplication by a scalar from
$K^*$, indicator function $f$ for $P_i$ in
$K\Delta_\prec(I)$, and $f$ is unique if $f(P_i)=1$.

{\rm (d)} If $\mathfrak{p}_i$ is the vanishing ideal of
$P_i$, then $\dim_K((I\colon\mathfrak{p}_i)/I)=1$ and
$(I\colon\mathfrak{p}_i)/I=K\overline{f}$ for any indicator function
$f$ for $P_i$, where $\overline{f}=f+I$.
\end{lemma}

\begin{proof} (a): The polynomial $h=g(P_i)f-f(P_i)g$
 vanishes at all points of $X$, that is, $h\in I$. If $h\neq 0$, then
the initial monomial of $h$ is in the initial ideal of $I$, a contradiction
since all monomials of $h$ are standard. Thus, $h=0$ and
$g(P_i)f=f(P_i)g$. 

(b): Note the equalities
$I=\bigcap_{j=1}^m\mathfrak{p}_j$ (Lemma~\ref{primdec-ixx}) and
$(I\colon\mathfrak{p}_i)=\bigcap_{j\neq i}\mathfrak{p}_j$. Let $f$ be
an indicator function for $P_i$, then $f\not\in\mathfrak{p}_i$ and $f\in\mathfrak{p}_j$ 
for $j\neq i$. Thus,
$f\in(I\colon\mathfrak{p}_i)$ and $f\notin I$. Conversely, take $f$ in
 $(I\colon\mathfrak{p}_i)\setminus I$. Then, $f\in \bigcap_{j\neq
i}\mathfrak{p}_j$ and $f\notin\mathfrak{p}_i$. Thus, $f$ is an
indicator function for $P_i$.

(c): The existence of $f$ follows from the division algorithm \cite[Theorem~3,
p.~63]{CLO} and part (b) because $(I\colon\mathfrak{p}_i)\setminus
I\neq\emptyset$. The uniqueness of $f$ follows
from part (a).

(d): Let $f$ be an indicator function for $P_i$. By part (b), $f$ 
is in $(I\colon\mathfrak{p}_i)\setminus I$. Therefore,
one has $(I\colon\mathfrak{p}_i)/I\supset K\overline{f}$.  
To show the other inclusion take
$\overline{0}\neq\overline{g}\in(I\colon\mathfrak{p}_i)/I$ and note
that $g$ is an indicator function for $P_i$ by part (b). 
By the division algorithm \cite[Theorem~3, p.~63]{CLO}, one has 
\begin{equation*}
(I\colon\mathfrak{p}_i)/I=\{\overline{h}\mid h\in
(I\colon\mathfrak{p}_i)\textstyle\bigcap K\Delta_\prec(I)\}.
\end{equation*}
\quad Hence, using part (b), we may assume that $f$ is an indicator
function for $P_i$ in $K\Delta_\prec(I)$, and we can
write $\overline{g}=\overline{h}$ for some $h$ in 
$(I\colon\mathfrak{p}_i)\bigcap K\Delta_\prec(I)$. By part (b), $h$ is
an indicator function for $P_i$. Hence, by part (a), we get 
$\overline{g}=\overline{h}=\lambda\overline{f}$,
$\lambda=h(P_i)/f(P_i)$. Thus, $\overline{g}\in K\overline{f}$.
\end{proof}

The following numerical invariant will
be used to determine the regularity index of the minimum distance 
of a Reed--Muller-type code (Proposition
\ref{tuesday-afternoon}). 
\begin{definition}
\label{def:v-number}
The v-\textit{number} of $I=I(X)$, denoted ${\rm v}(I)$,
is given by
$$
{\rm v}(I):=\min\{d\geq 0 \mid \text{there is $0\neq f \in S$, $\deg(f)=d$, and 
$\mathfrak{p} \in {\rm Ass}(I)$ with $(I\colon f) =\mathfrak{p}$}\},$$
where ${\rm Ass}(I)$ is the set of associated primes of $S/I$. 
\end{definition}
The v-number is finite by the definition of associated primes and
${\rm v}(I)\geq 1$ because $|X|\geq 2$. Let
$\mathfrak{p}_1,\ldots,\mathfrak{p}_m$ be the associated primes of
$I$, that is, $\mathfrak{p}_i$ is the vanishing 
ideal $I_{P_i}$ of $P_i$. One
can define the v-number of $I$ at each $\mathfrak{p}_i$ by
$$
{\rm v}_{\mathfrak{p}_i}(I):=\mbox{min}\{d\geq 0 \mid
\exists\, 0\neq f\in S,\ \deg(f)=d,\mbox{ with }(I\colon f)=\mathfrak{p}_i\}.
$$ 
\begin{lemma}\label{if-new}
The least degree of an indicator function for $P_i$ is equal to ${\rm
v}_{\mathfrak{p}_i}(I)$. 
\end{lemma}
\begin{proof}
A polynomial $f\in S$ is an indicator function 
for $P_i$ if and only if $(I\colon f)=\mathfrak{p}_i$. 
This follows using that the primary decomposition of $I$ is given by 
$I=\bigcap_{j=1}^m\mathfrak{p}_j$ (Lemma~\ref{primdec-ixx}) and 
noticing that $(I\colon f)=\bigcap_{f\notin \mathfrak{p}_j}\mathfrak{p}_j$
Hence, ${\rm v}_{\mathfrak{p}_i}(I)$ is the minimum degree of an
indicator function for $P_i$. 
\end{proof}

Note that the {\rm v}-number of $I$ is equal
to 
${\rm min}\{{\rm v}_{\mathfrak{p}_i}(I)\}_{i=1}^m$. 
To compute the v-number using {\it Macaulay\/}$2$ \cite{mac2} 
(Example~\ref{Hiram-example}), we give a description
for the v-number of $I$ using initial degrees of 
certain ideals of the quotient ring $S/I$. 

For an ideal $M\neq 0$ of $S/I$, we define 
$\alpha(M):=\min\{\deg(f)\mid\overline{f}\in M, f\notin I\}$.  
The next result was shown in
\cite[Proposition~4.2]{min-dis-generalized} for unmixed graded
ideals. For vanishing ideals we prove that the graded assumption is not needed.  

\begin{proposition}\label{lem:vnumber}
Let $I\subset S$ be the vanishing ideal of $X$. Then 
$$
{\rm v}(I)=\min\{\alpha\left((I\colon\mathfrak{p})/{I}\right) \mid
\mathfrak{p}\in{\rm Ass}(I)\},
$$
and $\alpha\left((I\colon\mathfrak{p})/{I}\right)={\rm
v}_{\mathfrak{p}}(I)$ for
$\mathfrak{p}\in{\rm Ass}(I)$.
\end{proposition}

\begin{proof} Let $\mathfrak{p}_1,\ldots,\mathfrak{p}_m$ be the
associated primes of $I$. If $I$ is the
vanishing ideal of a point, then $(I\colon 1)=I$,
$(I\colon I)=S$, and ${\rm v}(I)=\alpha(S/I)=0$. Thus, we may assume
that $X$ has at least two points. Since 
${\rm v}(I)=\min\{{\rm v}_{\mathfrak{p}_i}(I)\}_{i=1}^m$, we need only
show that $\alpha\left((I\colon\mathfrak{p})/{I}\right)$ is equal to 
${\rm v}_{\mathfrak{p}}(I)$ for $\mathfrak{p}\in{\rm Ass}(I)$. Fix $1\leq k\leq m$.
There is $f\in S$ such that $(I\colon f)=\mathfrak{p}_k$ and ${\rm
v}_{\mathfrak{p}_k}(I)=\deg(f)$. Then, $f\in(I\colon\mathfrak{p}_k)\setminus I$ and 
$$
{\rm
v}_{\mathfrak{p}_k}(I)=\deg(f)\geq\alpha((I\colon\mathfrak{p}_i)/I).
$$
\quad Since $I=\bigcap_{i=1}^m\mathfrak{p}_i\subsetneq\bigcap_{i\neq
k}\mathfrak{p}_i=(I\colon\mathfrak{p}_k)$, we can
pick a polynomial $g$ in
$(I\colon\mathfrak{p}_k)\setminus I$ such that
$\alpha((I\colon\mathfrak{p}_k)/I)=\deg(g)$. Note that
$g\not\in\mathfrak{p}_k$ since $g\notin I$. Therefore, from the inclusions
$$  
\mathfrak{p}_k\subset(I\colon
g)=\bigcap_{i=1}^m(\mathfrak{p}_i\colon
g)=\bigcap_{g\not\in\mathfrak{p}_i}\mathfrak{p}_i\subset\mathfrak{p}_k, 
$$
we get $(I\colon g)=\mathfrak{p}_k$, and consequently ${\rm
v}_{\mathfrak{p}_k}(I)\leq\deg(g)\leq\alpha((I\colon\mathfrak{p}_k)/I)$.
\end{proof}

By the next result, for each $P_i$ in $X$ there is a unique
indicator function $f_i$ for $P_i$ in $K\Delta_\prec(I)$ of degree
${\rm v}_{\mathfrak{p}_i}(I)$ satisfying $f_i(P_i)=1$. We call 
$F=\{f_1,\ldots,f_m\}$ the set of \textit{standard indicator functions} for $X$
(Example~\ref{Hiram-example}).  

\begin{proposition}\label{indicator-function-prop} Let $X=\{P_1,\ldots,P_m\}$ be a subset of
$K^s$, let $I=I(X)$ be its vanishing ideal, and let $\prec$ be a graded
monomial order on $S$. The following hold.
\begin{enumerate}
\item[(a)] For each $1\leq i\leq m$ there is a unique $f_i$ in 
$K\Delta_\prec(I)$ such that $f_i(P_i) =1$ and $f_i(P_j)=0$ if  
$j\neq i$. The degree of $f_i$ is ${\rm v}_{\mathfrak{p}_i}(I)$ and
the set  $F=\{f_1,\ldots,f_m\}$ is a $K$-basis for 
$K\Delta_\prec(I)$. 
\item[(b)] 
${\rm ker}(\varphi)=K\{f_i-f_m\}_{i=1}^{m-1}+I$.
\item[(c)] $K\{f_i-f_m\}_{i=1}^{m-1}={\rm
ker}(\varphi)\bigcap K\Delta_\prec(I)=K^\perp$.
\item[(d)] If $r_0={\rm reg}(H_I^a)$, then $\deg(f_i)\leq r_0$ for all $i$ and $\deg(f_j)=r_0$ for
some $j$.
\end{enumerate}
\end{proposition}

\begin{proof} 
(a): The existence and uniqueness of $f_i$ follows from
Lemma~\ref{if}. By Proposition~\ref{lem:vnumber}, we can pick 
$g\in(I\colon\mathfrak{p}_i)\setminus I$ such
that ${\rm
v}_{\mathfrak{p}_i}(I)=\alpha\left((I\colon\mathfrak{p}_i)/{I}\right)=\deg(g)$. As $\prec$ is graded, by
the division algorithm \cite[Theorem~3, 
p.~63]{CLO}, we can write $g=h+r_g$ for some $h\in I$ and some $r_g\in
K\Delta_\prec(I)$ with $\deg(r_g)\leq\deg(g)$. Noticing that $r_g\in(I\colon\mathfrak{p}_i)\setminus
I$, we get $\deg(r_g)=\deg(g)$. Since $r_g$ is an indicator function
for $P_i$ in $K\Delta_\prec(I)$, by Lemma~\ref{if}(a), we get $r_g=\lambda f_i$ for some
$\lambda\in K^*$. Thus $\deg(f_i)={\rm v}_{\mathfrak{p}_i}(I)$. To
show that $F$ is linearly independent assume that 
$\sum_{i=1}^m\lambda_if_i=0$ for some $\lambda_1,\ldots,\lambda_m$ in $K$. Hence,
evaluating both sides of 
this equality at each $P_j$ gives $\lambda_j=0$. Now, the dimension
of the linear space $K\Delta_\prec(I)$ is $m=|X|$ because $|\Delta_\prec(I)|=|X|$
(Lemma~\ref{sep11-20}). Thus, $F$ is a $K$-basis for
$K\Delta_\prec(I)$.

(b): First we show the inclusion ``$\supset$''. Clearly
$\ker(\varphi)\supset I$. Thus, this inclusion follows by noticing
that $f_i-f_m$ is in the kernel
of $\varphi$ since $\varphi(f_i-f_m)=f_i(P_i)-f_m(P_m)=0$. To show the
inclusion ``$\subset$'' take $f\in\ker(\varphi)$. By the division
algorithm, we can write $f=h+r_f$ for some $h\in I$ and
$r_f\in K\Delta_\prec(I)$. By part (a) we can write
$r_f=\sum_{i=1}^m\lambda_if_i$ for some $\lambda_i$'s in $K$. Then, by noticing
that $r_f\in\ker(\varphi)$, we get
$\varphi(r_f)=\sum_{i=1}^m\lambda_i=0$, and consequently
\begin{align*}
&r_f=\left(\sum_{i=1}^{m-1}\lambda_if_i\right)+\lambda_mf_m=\sum_{i=1}^{m-1}\lambda_i(f_i-f_m).
\end{align*}
\quad Thus, $f=h+r_f\in I+K\{f_i-f_m\}_{i=1}^{m-1}$.

(c): This follows from the proof of part (b).

(d): By Lemma~\ref{sep11-20} one has the inclusion $K\Delta_\prec(I)\subset S_{\leq
r_0}$. Thus, $\deg(f_i)\leq r_0$ for all $i$. Note that
$H_I^a(r_0-1)<H_I^a(r_0)$ by definition of $r_0={\rm reg}(H_I^a)$.
Then, by Lemmas~\ref{lemma-referee1} and \ref{sep11-20}, we obtain
$\Delta_\prec(I)\not\subset S_{\leq r_0-1}$. 
Hence, we can pick $t^a\in
\Delta_\prec(I)$ of degree $r_0$. Then, by part (a), $t^a$ is in 
$K\{f_1,\ldots,f_m\}$, and consequently $\deg(f_j)=r_0$ for some $j$.
\end{proof}

\begin{remark}\label{nov23-20} 
Using Lemma~\ref{if} and the ideal
$(I\colon\mathfrak{p}_i)/{I}$ of
Proposition~\ref{lem:vnumber}, we obtain an algebraic method to
compute the standard indicator functions for $X$ (Example~\ref{Hiram-example},
Procedure~\ref{8points-in-A3-procedure}). 
\end{remark}

\begin{remark}\label{r:matrix-indicator-function}
It is convenient to have a
matrix interpretation of the standard 
indicator functions for $X$. Recall that the
evaluation map defines an isomorphism of vector spaces 
$${\rm ev}\colon K\Delta_\prec(I)\to K^{m}.$$
\quad The standard monomials
$\Delta_\prec(I)$ ordered using $\prec$ form a basis for
$K\Delta_\prec(I)$. Also an order of the points in $X=\{P_1,\dots,
P_m\}$ defines the standard basis for  $K^m$. We let $M_{\rm ev}$ be the matrix
of the evaluation map in these bases, i.e. the $i$-th column of $M_{\rm ev}$
consists of the values of the $i$-th standard monomial at
$P_1,\dots,P_m$. Then the $i$-th column of the inverse $M_{\rm ev}^{-1}$
consists of the coefficients, relative to $\Delta_\prec(I)$, 
of the standard indicator function $f_i$ for $P_i$.
\end{remark}

\section{Duality of standard monomial codes}\label{S:duality-standard-monomial-codes}

We continue with our original setup where $K$ is a finite field, $X$ a subset of $K^s$ of size at least two, and  $I=I(X)$  the
vanishing ideal of $X$. Fix a monomial order $\prec$ and let $\Delta_\prec(I)$ 
be the corresponding set of standard monomials. In this section we consider {\it standard monomial codes},
that is evaluation codes defined by subspaces which have a basis of standard monomials.

\begin{definition}\label{standard-monomial-codes-def}
Given a subset $\Gamma\subset\Delta_\prec(I)$, let $\cL(\Gamma)$ be the $K$-span of the set
of all monomials $u\in\Gamma$. Then $\cL(\Gamma)_X$
is called the {\it standard monomial code} corresponding to $\Gamma$.
\end{definition}

Consider two standard monomial codes $\cL(\Gamma_1)_X$ and $\cL(\Gamma_2)_X$ for
some $\Gamma_1,\Gamma_2\subset\Delta_\prec(I)$. The main result of this
section (Theorem~\ref{T:combinatorial-condition} below) is a 
 combinatorial condition for  $\cL(\Gamma_1)_X$ to be monomially
 equivalent  to the dual ${\cL(\Gamma_2)_X}^\perp$. In what follows $\Gamma_1\Gamma_2$ denotes 
 the pair-wise product of the subsets, i.e.
 $$\Gamma_1\Gamma_2=\{u_1u_2\in S \mid u_1\in\Gamma_1, u_2\in\Gamma_2\}.$$
 Also, we write $\overline{u}\in K\Delta_\prec(I)$ for the representative of $u\in S$ modulo the ideal $I$.
 
First, recall the definition of monomial equivalence of codes and its properties.

\begin{definition}\label{equivalent-codes-def}
We say that two linear codes $C_1,C_2$ in $K^m$
are \textit{monomially equivalent} if there is $\beta=(\beta_1,\ldots,\beta_m)$
in $K^m$ such that $\beta_i\neq 0$ for all $i$ and 
$C_2 = \beta\cdot C_1=\{\beta\cdot c\mid c\in C_1\}$, 
where $\beta\cdot c$ is the vector given by
$(\beta_1c_1,\ldots,\beta_mc_m)$ for 
$c=(c_1,\ldots,c_m)\in C_1$. 
\end{definition}

\begin{remark}\label{dual-equation}
Monomial equivalence of codes is an equivalence relation. 
If $C_2=\beta\cdot C_1$, then
$C_1=\beta^{-1}\cdot C_2$, $(\beta\cdot
C_2)^\perp=\beta^{-1}\cdot C_2^\perp$, and if $C_1^\perp=\beta\cdot C_2$, then
$C_2^\perp=\beta\cdot C_1$. 
\end{remark}

To state the main result we will need the following definition.  We say a standard monomial
$t^e\in\Delta_\prec(I)$ is {\it essential} if it appears in each
standard indicator function of $X$ (cf. Proposition~\ref{indicator-function-prop}).  

\begin{theorem}\label{T:combinatorial-condition} 
Let $X$ be a subset of $K^s$ of size $m=|X|\geq 2$, and
$I=I(X)$ be the vanishing ideal of $X$. Fix a monomial order $\prec$ 
and let $t^e\in\Delta_{\prec}(I)$ be 
essential. 
Then for any $\Gamma_1,\Gamma_2\subset\Delta_\prec(I)$
satisfying  
\begin{enumerate}
\item $|\Gamma_1|+|\Gamma_2|=|X|$,
\item $t^e$ does not appear in $\overline{u}$ for any $u\in\Gamma_1\Gamma_2$,
\end{enumerate}
we have 
$\beta\cdot\cL(\Gamma_1)_X={\cL(\Gamma_2)_X}^\perp,$
for some $\beta=(\beta_1,\dots, \beta_m)\in(K^*)^m$. Moreover, $\beta_i$ is the coefficient of $t^e$ in the $i$-th standard indicator
function~$f_i$, for $i=1,\dots,m$. 
\end{theorem}

\begin{proof}
Let  $\beta\in(K^*)^m$ be as in the statement of the theorem.
Then, by Remark~\ref{r:matrix-indicator-function},  
$\beta$ is the last row of the matrix $M_{\rm ev}^{-1}$. Clearly $\beta$   
is  orthogonal to all but the last column of $M_{\rm ev}$. Since ${\rm
ev}$ is an isomorphism, there is a unique polynomial $g\in K\Delta_\prec(I)$ such that  
$g(P_i)=\beta_i$ for $P_i\in X$, $i=1,\dots,m$. Therefore, the orthogonality
property is equivalent to $\varphi(gt^a)=0$ for any $t^a\in
\Delta_\prec(I)$, $t^a\neq t^e$. By linearity, this implies  
\begin{equation}\label{e:orthogonal}
\varphi(gf)=0\ \text{ for any }\ f\in
K\left(\Delta_\prec(I)\setminus\{t^e\}\right).
\end{equation}

 Now pick any $h_i\in\cL(\Gamma_i)$, for $i=1,2$. 
 Note that
 the monomials appearing in $h_1h_2$ belong to $\Gamma_1\Gamma_2$. Let
 $f=\overline{h_1h_2}\in K\Delta_\prec(I)$ be the representative of $h_1h_2$ modulo $I$. Then,
 according to condition (2) above, the essential monomial $t^e$ does not appear in $f$.
Therefore, we have    
$$\varphi(gh_1h_2)=\varphi(gf)=0,$$
by Eq.~\eqref{e:orthogonal}. This shows that $g\cL(\Gamma_1)\subset 
(\ker(\varphi)\colon\cL(\Gamma_2))$. Applying the evaluation map to both sides and using  Theorem~\ref{formula-dual}, we obtain
$$\beta\cdot\cL(\Gamma_1)_X\subset{\cL(\Gamma_2)_X}^\perp.$$

 Finally, condition (1) ensures that the above inclusion is equality, as 
$$\dim_K\left(
\beta\cdot\cL(\Gamma_1)_X\right)=\dim_K(\cL(\Gamma_1)_X)=|\Gamma_1|=|X|-|\Gamma_2|=\dim_K
({\cL(\Gamma_2)_X}^\perp).$$
Note that the first equality holds since $\beta_i\neq 0$ for all $i$, as $t^e$ is essential.
\end{proof}

\begin{remark}\label{R:combinatorial-condition}  Note that in the case when $\Gamma_1\Gamma_2$ is contained in $\Delta_\prec(I)$, condition (2)
in the statement of Theorem~\ref{T:combinatorial-condition} can be relaxed to 
$\Gamma_1\Gamma_2\subset \Delta_\prec(I)\setminus\{t^e\}$.
\end{remark}

Duality formulas for certain toric complete intersection codes are
given in  \cite[Theorem~3.3]{Celebi-Soprunov}. These codes are a
generalization of projective evaluation codes on complete intersections. 

\section{A duality criterion for Reed--Muller-type
codes}\label{duality-criterion-rm}
 
In this section we concentrate on the class of evaluation codes, 
called Reed--Muller-type codes, defined by evaluating 
the subspace of polynomial of total degree up to $d$ at a set of
points $X\subset K^s$. As before we assume $X=\{P_1,\dots, P_m\}$ where $m\geq 2$.
The main result of this section is a duality criterion for Reed--Muller-type codes. It is then applied to 
the case when the vanishing ideal $I=I(X)$ is Gorenstein.

\begin{definition}\label{Reed-Muller-type-codes-def}\cite{duursma-renteria-tapia,GRT} 
Fix a degree $d\geq 1$ and let $S_{\leq 
d}=\bigoplus_{i=0}^dS_i$ be the
$K$-linear subspace of $S$ of all polynomials of degree at most $d$. If
$\mathcal{L}=S_{\leq d}$ then the resulting evaluation
code $\mathcal{L}_X$ is called a \textit{Reed--Muller-type code} of degree $d$ on $X$
and is denoted by $C_X(d)$. 
\end{definition}

The minimum distance of $C_X(d)$ is simply
denoted by $\delta_X(d)$. As is seen below 
the v-number of $I=I(X)$ is related to the asymptotic behavior of
$\delta_X(d)$ for $d\gg 0$.  There is $n\in\mathbb{N}$ such that 
$$
|X|=\delta_X(0)>\delta_X(1)>\cdots>\delta_X(n-1)>\delta_X(n)=\delta_X(d)=1\
\mbox{ for }\ d\geq n,
$$
see Proposition~\ref{behavior-hilbert-function}. The number $n$, denoted ${\rm reg}(\delta_X)$, is called the 
\textit{regularity index} of $\delta_X$. 

\begin{proposition}\label{tuesday-afternoon}
If $I=I(X)$, then ${\rm v}(I)={\rm reg}(\delta_X)\leq{\rm reg}(H_I^a)$.
\end{proposition}

\begin{proof} By of the Singleton bound for linear codes 
\cite[p.~71]{Huffman-Pless}, $\delta_X(d)=1$
for $d\geq{\rm reg}(H_I^a)$. Thus, ${\rm reg}(\delta_X)\leq{\rm
reg}(H_I^a)$. By Lemma~\ref{if-new}, ${\rm v}(I)$ is the minimum
degree of the indicator functions of the points of the set $X$. Then,
there is a point $P_i$ and an indicator function $f$ for $P_i$ 
such that $n_0={\rm v}(I)=\deg(f)$. Thus, $\delta_X(n_0)=1$ and 
${\rm reg}(\delta_X)\leq {\rm v}(I)$. If $n={\rm reg}(\delta_X)<{\rm
v}(I)$, then $\delta_X(n)=1$, and consequently there is $g\in S_{\leq n}$
and there is a point $P_j$ such that $g$ is an indicator function for $P_j$, a
contradiction because $n_0={\rm v}(I)\leq\deg(g)\leq n$.
\end{proof}

%

\begin{lemma}\label{aug1-20} Let $I$ be the vanishing ideal of $X$, 
let $r_0$ be the regularity index of
$H_I^a$, and let 
$\mathcal{C}$ be the set $S_{\leq r_0-1}\bigcap K\Delta_\prec(I)$. The following are equivalent.
\begin{enumerate} 
\item[(a)] $\mathcal{C}^\perp:=(\ker(\varphi)\colon\mathcal{C})\bigcap
K\Delta_\prec(I)=Kg$ for some $0\neq g\in S$. 
\item[(b)] $C_X(r_0-1)^\perp=K(g(P_1),\ldots,g(P_m))$ for some $0\neq
g\in K\Delta_\prec(I)$.
\item[(c)] $H_I^a(r_0-1)+1=|X|$.
\end{enumerate} 
\end{lemma}

\begin{proof} (a) $\Rightarrow$ (b): Under the evaluation map ``${\rm
ev}$'' the left (resp. right) hand side of the equality 
$\mathcal{C}^\perp=Kg$ map onto $C_X(r_0-1)^\perp$ (resp.
$K(g(P_1),\ldots,g(P_m))$).

(b) $\Rightarrow$ (c):
$1=\dim_K(C_X(r_0-1)^\perp)=|X|-\dim_K(C_X(r_0-1))=|X|-H_I^a(r_0-1)$.

(c) $\Rightarrow$ (a): Using Proposition~\ref{dual-properties}
together with the equality $\mathcal{C}=K(S_{\leq 
r_0-1}\bigcap\Delta_\prec(I))$ and Lemma~\ref{lemma-referee1}, we get
$\dim_K(\mathcal{C}^\perp)=|X|-\dim_K(\mathcal{C})=|X|-H_I^a(r_0-1)=1$. Hence,
$\dim_K(\mathcal{C}^\perp)=1$, and $\mathcal{C}^\perp=Kg$ for some
$0\neq g\in S$.
\end{proof}

\begin{proposition}\label{aug2-20} Let $I$ be the vanishing ideal of
$X$ and let $r_0$ be the regularity index of
$H_I^a$. Then, there exists $g\in
K\Delta_\prec(I)$ such that 
$$C_X(r_0-1)^\perp=K(g(P_1),\ldots,g(P_m))$$
and $g(P_i)\neq 0$ for all $i$ if and only if
$H_I^a(r_0-1)+1=|X|$ and ${\rm v}_{\mathfrak{p}_i}(I)=r_0$ for
all $i$. 
\end{proposition}
\begin{proof}
By Lemma~\ref{aug1-20}, $C_X(r_0-1)^\perp=K(g(P_1),\ldots,g(P_m))$
if and only if $H_I^a(r_0-1)+1=|X|.$ Note that the latter is equivalent to
having exactly one standard monomial of degree $r_0$ in $\Delta_\prec(I)$, by Lemma~\ref{lemma-referee1}.
Thus we only need to prove that
$g(P_i)\neq 0$ for all $i$ if and only if ${\rm v}_{\mathfrak{p}_i}(I)=r_0$ for all $i$.
Recall that ${\rm v}_{\mathfrak{p}_i}(I)$ equals the degree of the $i$-th standard indicator function $f_i$, see Proposition~\ref{indicator-function-prop}. 
Let $M_{\rm ev}$ be the matrix defined in Remark~\ref{r:matrix-indicator-function}.
Since $M_{\rm ev}^{-1}M_{\rm ev}=I_m,$ the $m$-th row $(c_1,\ldots,c_m)$ of
$M_{\rm ev}^{-1}$ is orthogonal to the first $m-1$ columns of $M_{\rm ev}$, i.e. 
$(c_1,\ldots,c_m)$ is orthogonal to $C_X(r_0-1)$.
By Remark~\ref{r:matrix-indicator-function}, 
$c_i$ is the coefficient in $f_i$ of the monomial of degree $r_0.$
We obtain $g(P_i)\neq 0$ for all $i$ if and only if
$c_i \neq 0$ for all $i$, which happens if and only 
${\rm v}_{\mathfrak{p}_i}(I)=\deg(f_i)=r_0$ for all $i$.
%
%
%
%
%
%
\end{proof}
We come to one of our main results.

\begin{theorem}{\rm(Duality criterion)}\label{duality-criterion} 
Let $X$ be a subset of $K^s$, $|X|\geq 2$, let $I=I(X)$ be its
vanishing ideal, let $r_0$ be the regularity index of $H_I^a$, and let $\prec$ be a graded
monomial order. The following conditions are equivalent.
\begin{enumerate} 
\item[(a)] $C_X(d)$ is monomially equivalent to $C_X(r_0-d-1)^\perp$  for 
$-1\leq d\leq r_0$.  
\item[(b)] $H_I^a(d)+H_I^a(r_0-d-1)=|X|$ for $-1\leq d\leq r_0$ 
and $r_0={\rm v}_{\mathfrak{p}}(I)$ for $\mathfrak{p}\in{\rm Ass}(I)$.
\item[(c)] There is 
$g\in K\Delta_\prec(I)$ such that $g(P_i)\neq 0$ for all $i$ and
$$
C_X(r_0-d-1)^\perp=(g(P_1),\ldots,g(P_m))\cdot C_X(d)\ \text{ for
\ $-1\leq d\leq r_0$}.
$$
Moreover, one can choose $g=\sum_{i=1}^m{\rm lc}(f_i)f_i$, where $f_i$ is
the $i$-th standard indicator function and ${\rm lc}(f_i)$ is its leading coefficient.
\end{enumerate} 
\end{theorem}
\begin{proof} (a) $\Rightarrow$ (b): Since
$\dim_K(C_X(r_0-d-1)^\perp)=|X|-H_I^a(r_0-d-1)$ and
$\dim_K(C_X(d))=H_I^a(d)$, one has $H_I^a(r_0-d-1)+H_I^a(d)=|X|$
because equivalent codes have the same dimension. As $C_X(0)$ is
equivalent to $C_X(r_0-1)^\perp$, there is
$\beta=(\beta_1,\ldots,\beta_m)$ in $K^m$ such that $\beta_i\neq 0$ for all $i$
and $C_X(r_0-1)^\perp=\beta\cdot C_X(0)=K\beta$. The image of
$S_{\leq r_0}$ under the evaluation map ``${\rm ev}$'' is 
equal to $K^{|X|}$. This follows from the equalities $H_I^a(r_0)=\dim_K(S_{\leq r_0}/I_{\leq
r_0})=|X|$. Hence, by the division algorithm, it follows that there is $g\in K\Delta_\prec(I)$ such
that $\beta=(g(P_1),\ldots,g(P_m))$. Then, by
Proposition~\ref{aug2-20}, we get ${\rm v}_{\mathfrak{p}}(I)=r_0$ for
$\mathfrak{p}\in{\rm Ass}(I)$.

(b) $\Rightarrow$ (c): This implication follows from
Theorem~\ref{T:combinatorial-condition} if we let $\Gamma_1$ and $\Gamma_2$ be the set of
monomials in $\Delta_\prec(I)$ of degree up to $d$ and up to $r_0-d-1$, respectively.
Then we have
$\cL(\Gamma_1)=S_{\leq d}\bigcap K\Delta_\prec(I)$ and  $\cL(\Gamma_2)=S_{\leq
r_0-d-1}\bigcap K\Delta_\prec(I)$ and, consequently,
$$
\mathcal{L}(\Gamma_1)_X=C_X(d)\ \text{ and }\ 
\mathcal{L}(\Gamma_2)_X=C_X(r_0-d-1). 
$$
\quad The largest standard monomial $t^e$ with respect to $\prec$
has total degree $r_0$ (Lemma~\ref{sep11-20}). By
Proposition~\ref{duality-hilbert-function}(d), there exists a unique
standard monomial of degree $r_0$ that is equal to $t^e$. Now, the
condition $r_0={\rm v}_{\mathfrak{p}}(I)$ for all $\mathfrak{p}\in{\rm
Ass}(I)$ means that every standard indicator function $f_i$ has total
degree $r_0$ (Proposition~\ref{indicator-function-prop}(a)) and, 
hence, the leading monomial of every $f_i$ is equal to $t^e$. Thus, $t^e$ is
an essential monomial. 
Furthermore, condition (1) of Theorem~\ref{T:combinatorial-condition}
translates to
$H_I^a(d)+H_I^a(r_0-d-1)=|X|$ (Lemma~\ref{lemma-referee1}). Also,
since $d+(r_0-d+1)<r_0$ we see that $\Gamma_1\Gamma_2\subset \Delta_\prec(I)\setminus\{t^e\}$
and, hence, condition (2) of Theorem~\ref{T:combinatorial-condition} is also satisfied 
(see Remark~\ref{R:combinatorial-condition}).
Therefore, by Theorem~\ref{T:combinatorial-condition}, we get 
$$\beta\cdot\cL(\Gamma_1)_X={\cL(\Gamma_2)_X}^\perp,$$
where $\beta_i$ is the coefficient of $t^e$ in the $i$-th standard indicator
function~$f_i$ for $i=1,\dots,m$. Setting $g=\sum_{i=1}^m{\rm
lc}(f_i)f_i$, one has $g(P_i)={\rm lc}(f_i)=\beta_i$ for all $i$.

(c) $\Rightarrow$ (a): This follows from the definition of equivalent
codes.
\end{proof}


The next result is known for a complete intersection homogeneous
vanishing ideal $I$ \cite[Lemma~3]{sarabia7}. As an
application we show this result under weaker conditions. 

\begin{corollary}\label{geramita-gorenstein} Let $X$ be a subset of
$K^s$, let $I$ be its vanishing ideal, and let $r_0$ be the regularity
index of $H_I^a$. 
If $H_I^a(r_0-1)+1=|X|$ and $r_0={\rm v}_{\mathfrak{p}}(I)$ for
$\mathfrak{p}\in{\rm Ass}(I)$, 
then $\delta(C_X(r_0-1))=2$.
\end{corollary}

\begin{proof} By Theorem~\ref{duality-criterion}, $C_X(r_0-1)^\perp$ is equivalent
of $K(1,\ldots,1)=C_X(0)$. Hence, $C_X(0)^\perp$ is equivalent to
$C_X(r_0-1)$. Since $C_X(0)$ is a repetition code, it is easy to see that $C_X(0)^\perp$ has minimum distance 2.
Therefore,
$\delta(C_X(r_0-1))=2$.
\end{proof}

\begin{definition}\cite[p.~171]{geramita-cayley-bacharach}
Let $Y$ be a finite subset of a projective space $\mathbb{P}^{s}$ over
a field $K$ and let $I(Y)$ be its homogeneous vanishing ideal in $R=K[t_0,\ldots,t_{s}]$.
We say that $Y$ is a \textit{Cayley--Bacharach scheme}
(CB-scheme) if every hypersurface of degree less than ${\rm
reg}(R/I(Y))$ that contains all but one
point of $Y$ must contain all the points of $Y$.
\end{definition}

To show the following proposition we need a result of 
Geramita, Kreuzer and Robbiano
\cite[Corollary~3.7]{geramita-cayley-bacharach} about CB-schemes. 
In loc. cit. the field $K$ is assumed to be infinite
\cite[p.~165]{geramita-cayley-bacharach} but the result that we need
it is seen to be valid for finite fields. 

\begin{proposition}\label{inequality-dim-dual} Let $X$ be a subset of
$K^s$, let 
$I$ be its vanishing ideal, and let $r_0$ be the regularity index of $H_I^a$. If 
$r_0={\rm v}_\mathfrak{p}(I)$ for all $\mathfrak{p}\in{\rm Ass}(I)$,
then 
\begin{enumerate}
\item[(a)] $H_I^a(d)+H_I^a(r_0-d-1)\leq |X|$ for $0\leq d\leq r_0$, and
\item[(b)] $\dim_K(C_X(d)^\perp)\geq \dim_K(C_X(r_0-d-1))$ for $0\leq d\leq r_0$. 
\end{enumerate}
\end{proposition}

\begin{proof} (a): The ideal $I^h$ is the homogeneous
vanishing ideal $I(Y)$ of the set $Y=\{[P,1] \mid \, P\in X\}$ of
points in the projective 
space  $\mathbb{P}^s$ (Lemma~\ref{oct3-20}). One has
the equality $H_I^a(d)=H_{I^h}(d)$ for $d\geq 0$
\cite[Lemma~8.5.4]{monalg-rev}. 
Hence, by \cite[Corollary~3.7]{geramita-cayley-bacharach}, it
suffices to show that $Y$ is a CB-scheme. Let $f$ be a homogeneous
polynomial of $S[u]$ of degree less than $r_0$ and let $Q$ be a point in $X$ such that
$f(P,1)=0$ for $P\in X\setminus\{Q\}$. We need only show that
$f(Q,1)=0$. Assume that $f(Q,1)\neq 0$. Consider the polynomial
$g=f(t_1,\ldots,t_s,1)$. Then, $g$ is an indicator function for $Q$ of
degree less than $r_0$, and consequently ${\rm
v}_\mathfrak{q}(I)<r_0$, where $\mathfrak{q}=I_Q$, a contradiction.

(b): This follows from part (a) and the equality
$H_I^a(d)+\dim_K(C_X(d)^\perp)=|X|$.
\end{proof}

In the next theorem (Theorem~\ref{gorenstein-vnumber} below) we show that when the vanishing ideal $I(X)$ is Gorenstein then
all standard indicator functions $f_i$ have  the same degree which equals the regularity of $I(X)$. The proof is based on the following lemma 
which relates the Castelnuovo--Mumford regularity and the socle of Artinian rings. Recall that the {\it socle} of an Artinian positively graded algebra $N=\bigoplus_{d\geq 0} N_d$ is
$$\mathrm{Soc}(N)=(0: N_{+}), \quad\text{where }\ N_{+}=\bigoplus_{d> 0} N_d.$$

\begin{lemma}\cite[Lemma 5.3.3]{monalg-rev}\label{reg-socle} Let  $J$ be a homogeneous ideal in a polynomial ring $R$. If $R/J$ is Artinian then
the Castelnuovo--Mumford regularity  ${\rm reg}(R/J)$ equals the maximal degree of generators of ${\rm Soc}(R/J)$.
\end{lemma}
The statement can also be found in \cite{Fro3}. We will also need the following lemma.

\begin{lemma}\label{jul26-20} Let $X=\{P_1,\ldots,P_m\}$ be a subset of
$K^s$ and let $I=I(X)$ be its vanishing ideal. If $f\in S$ is an indicator function for $P_i$ of minimum
degree $d$, then its homogenization $f^h$, with respect to the variable
$u$, is not in the ideal
$(I^h,u)$ and $\deg(f)={\rm v}_{\mathfrak{p}_i}(I)$.  
\end{lemma}

\begin{proof} If $f^h=g+uh$ for some $g\in S[u]_d\bigcap I^h$, $h\in S[u]_{d-1}$. The
ideal $I^h$ is the homogeneous vanishing ideal of the set
$Y=\{[P_i,1]\}_{i=1}^m$ of projective points in $\mathbb{P}^s$
(Lemma~\ref{oct3-20}). Hence, setting 
$H=h(t_1,\ldots,t_s,1)$, we get $f(P_i)=f^h(P_i,1)=H(P_i)\neq 0$ and
$f(P_j)=f^h(P_j,1)=H(P_j)=0$ for $j\neq i$, a contradiction because
$\deg(H)<\deg(f)$. By Lemma~\ref{if-new}, $\deg(f)={\rm
v}_{\mathfrak{p}_i}(I)$.
\end{proof}

\begin{theorem}\label{gorenstein-vnumber}
Let $X=\{P_1,\ldots,P_m\}$ be a subset of
$K^s$ and let $f_i$ be the $i$-th standard indicator function for
 $P_i$. If $I=I(X)$ is a Gorenstein ideal and $\mathfrak{p}_i$ is the
vanishing ideal of $P_i$, then ${\rm
v}_{\mathfrak{p}_i}(I)=\deg(f_i)={\rm reg}(H_I^a)$.
\end{theorem}

\begin{proof} Let $M=S[u]/I^h$ and consider the Artinian ring $M/uM$. Then
$$\mathrm{Soc}(M/uM)=((I^h,u)\colon\mathfrak{m})/(I^h,u),$$
where $\mathfrak{m}$ is the maximal ideal $(t_1,\ldots,t_s,u)$.
We set $g_i=f_i^h$. By Lemma~\ref{jul26-20}, 
$g_i\notin(I^h,u)$.  We claim that $\overline{g_i}$, the class of $g_i$ modulo $(I^h,u)$, is in
${\rm Soc}(M/uM)$. Indeed, let $P_{ij}$ be the $j$-th coordinate of $P_i$. Then,
for any $1\leq j\leq s$, the polynomial $(t_j-P_{ij}u)g_i$ vanishes on the set of projective points
$Y=\{[P_i,1]\}_{i=1}^m$ and, hence, $t_jg_i\in(I^h,u)$ (see Lemma~\ref{oct3-20}).
Also, $ug_i\in (I^h,u)$, trivially. Thus, $\mathfrak{m}g_i\in(I^h,u)$ which shows the claim.

Now if $S[u]/I^h$ is Gorenstein, 
then the socle of $M/uM$ is a $K$-vector space of dimension $1$
\cite[Corollary~5.3.5]{monalg-rev} spanned by
$\overline{g_i}$.  By Lemma~\ref{reg-socle}, the Castelnuovo--Mumford regularity of $M/uM$ equals
the degree of $g_i$. It remains to note that since $u$ is regular in
$M$ (see Lemma~\ref{u-is-regular}) and $M$ is Cohen--Macaulay,
the Castelnuovo--Mumford regularity of $M/uM$ and $M$ are equal (cf. \cite[p. 175]{BH}). Therefore,
$${\rm v}_{\mathfrak{p}_i}(I)=\deg(f_i)=\deg(g_i)={\rm
reg}(M/uM)={\rm reg}(M)={\rm reg}(H_I^a),
$$ 
where the first equality follows from
Proposition~\ref{indicator-function-prop} and the last equality was
discussed before Lemma~\ref{oct3-20}.  
\end{proof}

\begin{remark}
Note that the above proof implies that $\mathrm{in}_\prec(f_i)=\mathrm{in}_\prec(f_m)$ and the class of
${\rm in}_\prec(f_i)$ modulo the ideal $(I^h,u)$ is in ${\rm Soc}(S[u]/(I^h,u))$ for all $i$. Indeed,
by Corollary~\ref{duality-hilbert-gorenstein}(b), there is only one
standard monomial of degree $r_0={\rm reg}(H_I^a)$. Hence ${\rm
in}_\prec(f_i)={\rm in}_\prec(f_m)$ because 
$f_i\in K\Delta_\prec(I)$. We may assume that $g_i$ is monic. Then, 
$\overline{g_i}=\overline{\mathrm{in}_\prec(f_i)}$ because
$f_i^h-\mathrm{in}_\prec(f_i)$ is equal to $uh_i$ for some $h_i\in S$, and 
$\overline{\mathrm{in}_\prec(f_i)}$ is in the socle of
$S[u]/(I^h,u)$.
\end{remark}

\begin{corollary}\label{sept1-18} Let $X$ be a 
set of points in $K^s$ and let $I=I(X)$ be its
vanishing ideal. If $I$ is Gorenstein, then
$\delta_X(d)\geq {\rm reg}(H_I^a)-d+1$ for $1\leq d<{\rm reg}(H_I^a)$.
\end{corollary}

\begin{proof} Let $r_0$ be the regularity of $H_I^a$. If $r_0=1$,
there is nothing to prove. Assume $r_0\geq 2$. By
Corollary~\ref{geramita-gorenstein}, $\delta_X(r_0-1)=2$. Hence,
by Proposition~\ref{behavior-hilbert-function}, we get
$\delta_X(d)\geq(r_0-1-d)+\delta_X(r_0-1)$. Thus, $\delta_X(d)\geq
r_0-d+1$. 
\end{proof}

\begin{corollary}\label{gorenstein-prop}
Let $X$ be a subset of $K^s$, let $I$ be its vanishing ideal, and
let $r_0$ be the regularity index of $H_I^a$.  
If $\mathcal{C}=S_{\leq r_0-1}\bigcap K\Delta_\prec(I)$ and $I$ is
Gorenstein,  
then there is $g\in S$ such that $\mathcal{C}^\perp=Kg$ and
$g(P_i)\neq 0$ for all $i$. 
\end{corollary}
\begin{proof} 
As $I$ is Gorenstein, by Corollary~\ref{duality-hilbert-gorenstein}, one has 
$1=H_I^a(0)=|X|-H_I^a(r_0-1)$ and, by
Theorem~\ref{gorenstein-vnumber}, one 
has $r_0={\rm reg}(H_I^a)={\rm v}_{\mathfrak{p}}(I)$ for all
$\mathfrak{p}\in{\rm Ass}(I)$.  Hence, the result follows from
Theorem~\ref{duality-criterion}.
\end{proof}

The following result can be applied to any Reed--Muller-type code
$C_X(d)$ when the vanishing ideal $I(X)$ is a complete intersection
generated by a Gr\"obner basis with $s=\dim(S)$ elements 
(Corollary~\ref{duality-hilbert-gorenstein}(c)). In
particular, since the vanishing 
ideal of a Cartesian
set is a complete intersection generated by a Gr\"obner basis with
$s$ elements \cite[Lemma~2.3]{cartesian-codes}, we recover the duality theorems for  
affine Cartesian codes given in \cite[Theorem~5.7]{GHWCartesian} and 
\cite[Theorem~2.3]{Lopez-Manganiello-Matthews}.

\begin{corollary}\label{gorenstein-codes} 
Let $X=\{P_1,\ldots,P_m\}$ be a subset of $K^s$,
$|X|\geq 2$, let $I=I(X)$ be its vanishing ideal, 
let $r_0$ be the regularity index of $H_I^a$, and let $\prec$ be a graded
monomial order on $S$. If $I$ is Gorenstein, then there is 
$g\in K\Delta_\prec(I)$ such that $g(P_i)\neq 0$ for all $i$ and
$$
(g(P_1),\ldots,g(P_m))\cdot C_X(r_0-d-1)=C_X(d)^\perp\ \text{ for
\ $-1\leq d\leq r_0$}.
$$
Moreover, one can choose $g=\sum_{i=1}^m{\rm lc}(f_i)f_i$, where $f_i$ is
the $i$-th standard indicator function and ${\rm lc}(f_i)$ is its leading coefficient.
\end{corollary}
\begin{proof} By Corollary~\ref{duality-hilbert-gorenstein}(a) and
Theorem~\ref{gorenstein-vnumber}, the two conditions of
Theorem~\ref{duality-criterion}(b) hold, and the result follows from
Theorem~\ref{duality-criterion}.   
\end{proof}

When $r_0$ is odd and ${\rm char}(K)=2$ this construction produces
self-dual Reed--Muller-type codes as the following Corollary shows. 
Recall that a linear code $C$ is {\it self-dual} if $C=C^\perp$.

\begin{corollary}\label{self-dual-gorenstein-codes} 
Let $X=\{P_1,\ldots,P_m\}$ be a subset of $K^s$,
$|X|\geq 2$, let $I=I(X)$ be its vanishing ideal, 
let $r_0$ be the regularity index of $H_I^a$, and let $\prec$ be a graded
monomial order. Assume ${\rm char}(K)=2$, $r_0$ is odd, and $I$ is Gorenstein.
Define $\alpha=(\alpha_1,\dots,\alpha_m)\in (K^*)^m$ by 
$\alpha_i^2={\rm lc}(f_i)$, where ${\rm lc}(f_i)$ is the leading
coefficient of the $i$-th 
standard indicator function $f_i$ for $P_i$. 
Then the linear code $\alpha\cdot C_X((r_0-1)/2)$ is self-dual.
\end{corollary}

\begin{proof} Setting $d=(r_0-1)/2$ in
Corollary~\ref{gorenstein-codes}, we have 
\begin{equation}\label{e:self}
\beta\cdot C_X((r_0-1)/2)=C_X((r_0-1)/2)^\perp,
\end{equation}
for some $\beta=(g(P_1),\ldots,g(P_m))\in (K^*)^m$. Recall that
we can choose $g$ such that
$\beta_i=g(P_i)={\rm lc}(f_i)$. As ${\rm char}(K)=2$, there exists
$\alpha_i\in K^*$ such that $\alpha_i^2=\beta_i$. Then Eq.~\eqref{e:self}
implies that for any $u,v\in C_X((r_0-1)/2)$,    
$$\langle \alpha\cdot u,\alpha\cdot v\rangle=\langle \beta\cdot u, v\rangle=0.$$
\quad Therefore, $\alpha\cdot C_X((r_0-1)/2)$ is contained in its
dual and, by Eq.~\eqref{e:self}, 
has the same dimension as its dual, i.e. is self-dual.  
\end{proof}

\section{The algebraic dual of monomial evaluation
codes}\label{dual-monomial-section}

In this section we study the dual and the algebraic dual of two families of
evaluation codes. If an evaluation code is monomial and the set of evaluation points is a 
degenerate torus (resp. degenerate affine space), we show that its 
algebraic dual is a monomial space (resp. we classify when
its algebraic dual is a monomial space). 

\subsection{Monomial evaluation codes on a degenerate torus} Let
$A_1,\ldots,A_s$ be  
subgroups of the multiplicative group $K^*$ of the finite
field $K=\mathbb{F}_q$, let 
$$T:=A_1\times\cdots\times A_s=\{P_1,\ldots,P_m\}$$
be the Cartesian product of $A_1,\ldots,A_s$, and let $\mathcal{L}_T$
be a monomial code on $T$, that is, $\mathcal{L}$
is generated by a finite set of monomials of $S$. The set $T$ is called 
a \textit{degenerate torus} \cite{cartesian-codes}. In this
subsection we determine the algebraic dual 
$\mathcal{L}^\perp$ and the dual $(\mathcal{L}_T)^\perp$ of $\mathcal{L}_T$ in terms of the generators of
$\mathcal{L}$ and show that $(\mathcal{L}_T)^\perp$ is a standard
monomial code on $T$. 

The order of the cyclic group $A_i$ is denoted by $d_i$ for
$i=1,\ldots,s$. Let $\prec$ be a graded monomial order on $S$. 
The vanishing ideal $I=I(T)$ is generated by the Gr\"obner 
basis $\mathcal{G}=\{t_i^{d_i}-1\}_{i=1}^{s}$
\cite[Lemma~2.3]{cartesian-codes}, and consequently $\Delta_\prec(I)$
is the set of all 
monomials $t^c$,
$c=(c_1,\ldots,c_s)$, such that $0\leq c_i\leq d_i-1$ for
$i=1,\ldots,s$. By Proposition~\ref{binomial-standard} and
Lemma~\ref{dual-standard}, the standard function space
$\widetilde{\mathcal{L}}$ of $\mathcal{L}_T$ is a monomial space of
$S$ and $\mathcal{L}^\perp=\widetilde{\mathcal{L}}^\perp$. Thus, we may
assume that $\mathcal{L}=\widetilde{\mathcal{L}}$.
Let
$\mathcal{A}=\{t^{a_1},\ldots,t^{a_k}\}\subset\Delta_\prec(I)$ be the unique 
monomial $K$-basis of $\mathcal{L}$ where 
$$
t^{a_i}=t_1^{a_{i,1}}\cdots t_s^{a_{i,s}},\
a_i=(a_{i,1},\ldots,a_{i,s})\in\mathbb{N}^s,\ \mbox{ for }\ i=1,\ldots,k.  
$$
\quad The {\it support} of $t^{a_i}$, denoted ${\rm supp}(t^{a_i})$, is the set of
all $t_j$ such that $a_{i,j}>0$. To construct a monomial basis for 
$\mathcal{L}^\perp$ we will need the following set of monomials. 
For each $1\leq i\leq k$, we set $t^{b_i}:=1$ if
$t^{a_i}=1$ and 
\begin{equation}\label{jul6-1-20}
t^{b_i}=t_1^{b_{i,1}}\cdots t_s^{b_{i,s}}:=\!\!\prod_{t_j\in{\rm
supp}(t^{a_i})}\!\!t_j^{d_j-a_{i,j}}\ \mbox{ if }\ t^{a_i}\neq 1.
\end{equation}
\quad The set $\mathcal{B}:=\{t^{b_1},\ldots,t^{b_k}\}$ has cardinality
$k$, $\mathcal{B}\subset\Delta_\prec(I)$, and the support of $t^{b_i}$ is
the support of $t^{a_i}$ for $i=1,\ldots,k$. The regularity index of
$H_I^a$ is $r_0=\sum_{i=1}^s(d_i-1)$ \cite[Proposition~2.5]{cartesian-codes} and 
$H_I^a(r_0)=|\Delta_\prec(I)|=|T|=d_1\cdots d_s$ (Lemma~\ref{sep11-20}). 

\begin{lemma}\label{jul6-20}
Let $t^c$, $c=(c_1,\ldots,c_s)$, be a monomial of $S$. 
If $c_i\equiv 0\ {\rm mod} (d_i)$ for 
all $i$, then $t^c\notin{\rm ker}(\varphi)$. If 
$c_i\not\equiv 0\ {\rm mod}(d_i)$ for some $i$, then 
$t^c\in{\rm ker}(\varphi)$. 
\end{lemma}

\begin{proof} Let $P_1,\ldots,P_m$ be the points of $T$ and let
$\beta_i$ be a  generator of the multiplicative  
cyclic group $A_i$ for $i=1,\ldots,s$. Recall that $q=p^v$ for some
prime number $p>0$ and $v\in\mathbb{N}_+$. Note that each $d_i$ is relatively prime 
to $p$ because $d_i=|A_i|$ divides $q-1=|K^*|=|(\mathbb{F}_q)^*|$. Thus,
$\gcd(m,p)=1$ because $m=|T|=d_1\cdots
d_s$. Assume that $c_i\equiv 0\ {\rm mod} (d_i)$ for 
all $i$. Then, $\varphi(t^c)=\sum_{i=1}^mt^c(P_i)= m\cdot 1$ and $m\cdot 1\neq 0$ because
$\gcd(m,p)=1$. Thus, $t^c\notin{\rm ker}(\varphi)$. Assume that $c_i\not\equiv 0\ {\rm mod}(d_i)$ for some
$i$. Then, $c_i\geq 1$. For simplicity of
notation assume that $i=1$. The cartesian set $T$ can be partitioned as
$$
T=\{P_1,\ldots,P_m\}=\bigcup_{i=1}^{d_1}\{(\beta_1^i,Q)\mid Q\in
A_2\times\cdots\times A_s\}. 
$$
\quad Hence, setting $T_1=A_2\times\cdots\times A_s$, we obtain  
$$
\varphi(t^c)=
\sum_{i=1}^mt^c(P_i)=\left(1+\beta_1^{c_1}+(\beta_1^2)^{c_1}+\cdots+
(\beta_1^{d_1-1})^{c_1}\right)\left(\sum_{Q\in T_1}t_2^{c_2}\cdots
t_s^{c_s}(Q)\right).
$$ 
\quad Hence, using the equality
$(\sum_{i=0}^{d_1-1}(\beta_1^{c_1})^i)(\beta_1^{c_1}-1)=(\beta_1^{c_1})^{d_1}-1=0$
and, noticing that $\beta_1^{c_1}-1=0$ if and only if $c_1\equiv 0\ {\rm
mod} (d_1)$, we get $\sum_{i=0}^{d_1-1}(\beta_1^{c_1})^i=0$. Thus,
$\varphi(t^c)=0$ and $t^c\in{\rm ker}(\varphi)$. 
\end{proof}

The main result in connection with monomial evaluation codes over $T$ is the following. 

\begin{proposition}{\rm(Monomial basis)}\label{dual-toric-degenerate} Let $\mathcal{L}$ be a subspace
with a basis of standard monomials $\mathcal{A}=\{t^{a_1},\dots,t^{a_k}\}$. Then
$\mathcal{L}^\perp=K(\Delta_\prec(I)\setminus\mathcal{B})$,
where $\mathcal{B}=\{t^{b_1},\dots,t^{b_k}\}$ is defined in Eq. (\ref{jul6-1-20}).
\end{proposition}

\begin{proof} As $|\Delta_\prec(I)|=|T|$ and $\mathcal{L}_X$ is
a standard evaluation code, by Proposition~\ref{dual-properties}, one has 
$$
\dim_K(\mathcal{L}^\perp)=|T|-\dim_K(\mathcal{L})=|T|-k=|\Delta_\prec(I)\setminus\mathcal{B}|.
$$
\quad Thus, to show the equality
$\mathcal{L}^\perp=K(\Delta_\prec(I)\setminus\mathcal{B})$, we need only show that
$\Delta_\prec(I)\setminus\mathcal{B}\subset\mathcal{L}^\perp$. Take
$t^\alpha\in\Delta_\prec(I)\setminus\mathcal{B}$,
$\alpha=(\alpha_1,\ldots,\alpha_s)$. Since $\mathcal{L}$ is
generated by the set $\{t^{a_1},\ldots,t^{a_k}\}$ it suffices to show that
$t^\alpha t^{a_i}$ is in ${\rm ker}(\varphi)$ for $i=1,\ldots,k$.
Assume $t^{a_i}=1$. Then, $t^{b_i}=1$ and $t^\alpha\neq 1$. Hence, by
Lemma~\ref{jul6-20}, $1\cdot t^\alpha\in{\rm ker}(\varphi)$ since
$t^\alpha$ is a standard monomial of $S/I$ and $\alpha\neq 0$. Assume
that $t^{a_i}\neq 1$. If ${\rm supp}(t^\alpha)\not\subset{\rm
supp}(t^{a_i})$, then there is $t_j\in {\rm supp}(t^\alpha)$ and
$t_j\notin{\rm supp}(t^{a_i})$. Thus, $t_j^{\alpha_j}$ divides
$t^{a_i}t^\alpha$, $t_j^{\ell}$ does not divide 
$t^{a_i}t^{\alpha}$ for $\ell>\alpha_j$, and $\alpha_j\not\equiv 0\
{\rm mod}(d_j)$ because $\alpha_j\leq d_j-1$. Hence, by
Lemma~\ref{jul6-20}, we get
$t^{a_i}t^\alpha\in{\rm ker}(\varphi)$. The case ${\rm
supp}(t^\alpha)\not\supset{\rm
supp}(t^{a_i})$ can be treated similarly. Thus, we may assume
${\rm supp}(t^\alpha)={\rm supp}(t^{a_i})$. By definition of $t^{b_i}$
one also has ${\rm supp}(t^{b_i})={\rm supp}(t^{a_i})$. As
$t^{a_i}\neq 1$, by
Eq.~(\ref{jul6-1-20}), we obtain
$$
t^{a_i}t^\alpha=t_1^{a_{i,1}+\alpha_1}\cdots t_s^{a_{i,s}+\alpha_s}=\!\!\prod_{t_j\in{\rm
supp}(h_i)}\!\!t_j^{d_j-b_{i,j}+\alpha_j}.
$$ 
\quad If $c_j=d_j-b_{i,j}+\alpha_j\not\equiv 0\ {\rm mod} (d_j)$ for some
$j$, then $t^{a_i}t^\alpha$ is in ${\rm ker}(\varphi)$ by
Lemma~\ref{jul6-20}. 
If $c_j=d_j-b_{i,j}+\alpha_j\equiv 0\ {\rm mod} (d_j)$ for all $j$, then
it follows readily that $b_{i,j}=\alpha_j$ for all $j$, that is,
$t^{b_i}=t^\alpha$, a contradiction since $t^\alpha$ is not in
$\mathcal{B}$. 
\end{proof}

\begin{corollary}\label{dual-toric-degenerate-coro}
Let $\mathcal{L}_T$ be a monomial code on $T$ and let
$\mathcal{L}^\perp$ be its algebraic dual. Then, 
$(\mathcal{L}_T)^\perp=(\mathcal{L}^\perp)_T$  and $(\mathcal{L}_T)^\perp$ is a
standard monomial code on $T$.
\end{corollary}

\begin{proof}
The linear code $(\mathcal{L}_T)^\perp$ is the
evaluation code $(\mathcal{L}^\perp)_T$ on $T$ 
(Theorem~\ref{formula-dual}) and 
$\mathcal{L}^\perp$ is generated by the set of monomials
$\Delta_\prec(I)\setminus\mathcal{B}$
(Proposition~\ref{dual-toric-degenerate}). 
Thus, $(\mathcal{L}_T)^\perp$ is a standard monomial code on $T$
because its standard function space is $\mathcal{L}^\perp$.
\end{proof}

\begin{corollary}\cite{Bras-Amoros-O'Sullivan,Ruano}\label{dual-toric-coro}
Let $\mathcal{L}_T$ be a generalized toric code on $T=(K^*)^s$,
$K=\mathbb{F}_q$, and let
$\mathcal{L}^\perp$ be its algebraic dual. Then, 
$(\mathcal{L}_T)^\perp=(\mathcal{L}^\perp)_T$  and $(\mathcal{L}_T)^\perp$ is a
generalized toric code.
\end{corollary}

\begin{proof} It follows at once from
Corollary~\ref{dual-toric-degenerate-coro} by making
$A_i=K^*$ for $i=1,\ldots,s$.
\end{proof}

\subsection{Monomial evaluation codes on a degenerate affine space} Let
$K=\mathbb{F}_q$ be a finite field of characteristic $p$, let $A_1,\ldots,A_s$ be  
subgroups of the multiplicative group $K^*$ of the
field $K$, let $B_i$ be the set $A_i\cup\{0\}$ for $i=1,\ldots,s$, 
let  
$$\mathcal{X}:=B_1\times\cdots\times B_s=\{P_1,\ldots,P_m\}$$
be the Cartesian product of $B_1,\ldots,B_s$, and let
$\mathcal{L}_\mathcal{X}$
be a monomial code on $\mathcal{X}$, that is, $\mathcal{L}$ is generated
by a finite set of monomials of $S$. The set $\mathcal{X}$ is called  
a \textit{degenerate affine space}. In this
subsection we classify when the algebraic dual
$\mathcal{L}^\perp$ is generated by monomials. We also classify when
the dual $(\mathcal{L}_\mathcal{X})^\perp$ of 
$\mathcal{L}_\mathcal{X}$ is a standard monomial code, and show that in certain
interesting cases $(\mathcal{L}_\mathcal{X})^\perp$ is a standard
monomial code.

The order of the multiplicative monoid $B_i$ is denoted by $e_i$ and
the order of $A_i$ is denoted by $d_i$ for $i=1,\ldots,s$.  Let
$\prec$ be a graded monomial order on $S$.  
By \cite[Lemma 2.3]{cartesian-codes}, the vanishing ideal $I=I(\mathcal{X})$ of
$\mathcal{X}$ is generated by the Gr\"obner 
basis $\mathcal{G}=\{t_i^{e_i}-t_i\}_{i=1}^{s}$, and consequently the set of standard
monomials $\Delta_\prec(I)$ of $S/I$ is the set of all $t^c$,
$c=(c_1,\ldots,c_s)$, such that $0\leq c_i\leq d_i$ for
$i=1,\ldots,s$. By
Proposition~\ref{binomial-standard} and Lemma~\ref{dual-standard}, the standard function space
$\widetilde{\mathcal{L}}$ of $\mathcal{L}_X$ is a monomial space of
$S$ and
$\mathcal{L}^\perp=\widetilde{\mathcal{L}}^\perp$. Thus, we may
assume that $\mathcal{L}=\widetilde{\mathcal{L}}$. Note that $\mathcal{L}$ has a unique basis $\mathcal{A}$ 
consisting of standard monomials of $S/I$. We will classify when
$\mathcal{L}^\perp$ is a monomial space of $S$ and also when $(\mathcal{L}_\mathcal{X})^\perp$ 
is a standard monomial code in terms of this basis.
Let 
$\mathcal{A}=\{t^{a_1},\ldots,t^{a_k}\}\subset\Delta_\prec(I)$ be the unique
monomial $K$-basis of $\mathcal{L}$ where, as before, 
$$
t^{a_i}=t_1^{a_{i,1}}\cdots t_s^{a_{i,s}},\
a_i=(a_{i,1},\ldots,a_{i,s})\in\mathbb{N}^s,\ \mbox{ for }\ i=1,\ldots,k.  
$$
To construct a candidate for a
basis of $\mathcal{L}^\perp$, for each $1\leq i\leq k$, we set 
\begin{equation}\label{e:set-b}
t^{b_i}=t_1^{b_{i,1}}\cdots t_s^{b_{i,s}}:=\prod_{j=1}^st_j^{d_j-a_{i,j}}.
\end{equation}
\quad The set $\mathcal{B}:=\{t^{b_1},\ldots,t^{b_k}\}$ has cardinality
$k$ and $\mathcal{B}\subset\Delta_\prec(I)$. The index of regularity of
$H_I^a$ is $r_0=\sum_{i=1}^sd_i$ \cite[Proposition~2.5]{cartesian-codes} and 
$H_I^a(r_0)=|\Delta_\prec(I)|=|\mathcal{X}|=m=e_1\cdots e_s$ (Lemma~\ref{sep11-20}). 

\begin{lemma}\label{jul6-20-affine}
Let $t^c$, $c=(c_1,\ldots,c_s)$, be a monomial of $S$. The following
hold.
\begin{enumerate}
\item[(a)] If $c_i\not\equiv 0\ {\rm mod}(d_i)$ for some $i$, then 
$t^c\in{\rm ker}(\varphi)$. 
\item[(b)] If $p={\rm char}(K)$ and $\gcd(e_i,p)=p$ 
for some $i$, then $1\in{\rm ker}(\varphi)$. 
\item[(c)] If $\gcd(e_i,p)=p$ and $c_i\equiv 0\ {\rm mod}(d_i)$ for
all $i$, and $|{\rm supp}(t^c)|<s$, then $t^c\in{\rm ker}(\varphi)$.
\item[(d)] If $t^c=t_1^{\lambda_1d_1}\cdots t_s^{\lambda_sd_s}$ and
the $\lambda_i$'s are positive integers, then $t^c\notin{\rm ker}(\varphi)$.
\item[(e)] If $\gcd(e_i,p)=p$ for all $i$, then $\Delta_\prec(I)\bigcap{\rm
ker}(\varphi)=\Delta_\prec(I)\setminus\{t_1^{d_1}\cdots t_s^{d_s}\}$.
\end{enumerate}
\end{lemma}

\begin{proof} Let $P_1,\ldots,P_m$ be the points of $\mathcal{X}$,
$m=|\mathcal{X}|=e_1\cdots e_s$, and let
$\beta_i$ be a  generator of the multiplicative  
cyclic group $A_i$ for $i=1,\ldots,s$. 

(a): Assume that $c_i\not\equiv 0\ {\rm mod}(d_i)$ for some
$i$. Then, $c_i\geq 1$. For simplicity of
notation assume that $i=1$. The cartesian set $\mathcal{X}$ can be partitioned as
$$
\mathcal{X}=\left(\bigcup_{i=1}^{d_1}\{(\beta_1^i,Q) \mid Q\in
B_2\times\cdots\times B_s\}\right)\bigcup\{(0,Q) \mid  Q\in B_2\times\cdots\times B_s\}. 
$$
\quad Hence, setting $\mathcal{X}_1=B_2\times\cdots\times B_s$, we obtain  
$$
\varphi(t^c)=
\sum_{i=1}^mt^c(P_i)=\left(1+\beta_1^{c_1}+(\beta_1^2)^{c_1}+\cdots+
(\beta_1^{d_1-1})^{c_1}\right)\left(\sum_{Q\in \mathcal{X}_1}t_2^{c_2}\cdots
t_s^{c_s}(Q)\right).
$$ 
\quad Hence, using the equality
$(\sum_{i=0}^{d_1-1}(\beta_1^{c_1})^i)(\beta_1^{c_1}-1)=(\beta_1^{c_1})^{d_1}-1=0$
and, noticing that $\beta_1^{c_1}-1=0$ if and only if $c_1\equiv 0\ {\rm
mod} (d_1)$, we get $\sum_{i=0}^{d_1-1}(\beta_1^{c_1})^i=0$. Thus,
$\varphi(t^c)=0$ and $t^c\in{\rm ker}(\varphi)$. 

(b): Assume $\gcd(e_i,p)=p$ for some $i$. Then, 
$\varphi(1)=(e_1\cdots e_s)\cdot 1=0$. Thus, $1\in{\rm ker}(\varphi)$. 

(c): By part (b), we may assume ${\rm supp}(t^c)\neq\emptyset$, that
is, $t^c\neq 1$. For
simplicity of notation assume that 
${\rm supp}(t^c)=\{t_1,\ldots,t_\ell\}$, where $1\leq\ell<s$. For each
$1\leq i\leq \ell$, there is $\lambda_i\in\mathbb{N}_+$ such that
$c_i=\lambda_id_i$. We set $\lambda_i=0$ for $\ell<i\leq s$. The set $\mathcal{X}=\{P_1,\ldots,P_m\}$ 
can be partitioned as
$$
\mathcal{X}=A\textstyle\bigcup(\mathcal{X}\setminus A),\quad A=A_1\times\cdots\times A_\ell\times
B_{\ell+1}\times\cdots\times B_s.
$$
\quad Note that $t^c(P_i)=1$ if $P_i\in A$ and $t^c(P_i)=0$ if $P_i
\in \mathcal{X}\setminus A$. Hence
\begin{equation}\label{jul18-20}
\varphi(t^c)=
\sum_{i=1}^mt^c(P_i)=\sum_{i=1}^mt_1^{\lambda_1d_1}\cdots
t_s^{\lambda_sd_s}(P_i)=((d_1\cdots d_\ell)(e_{\ell+1}\cdots
e_s))\cdot 1.
\end{equation}
\quad Since $\ell<s$ and $\gcd(e_i,p)=p$ for all $i$, we get 
$(e_{\ell+1}\cdots e_s)\cdot 1=0$. Thus, $t^c\in{\rm ker}(\varphi)$.  

(d): As $\gcd(d_i,p)=1$ for all $i$, from Eq.~(\ref{jul18-20}), we get
$\varphi(t^c)=(d_1\cdots d_s)\cdot 1\neq 0$. 

(e): The inclusion ``$\subset$'' follows from part (d). To show the inclusion
``$\supset$'' take a monomial 
$t^c=t_1^{c_1}\cdots t_s^{c_s}$ in
$\Delta_\prec(I)\setminus\{t_1^{d_1}\cdots t_s^{d_s}\}$. Then, 
$c_i\leq d_i$ for all $i$ and $c_j<d_j$ for some $j$. We need only
show $t^c\in\ker(\varphi)$. By part (a), we may assume $c_i\equiv 0\
{\rm mod}(d_i)$ for all $i$. Hence $c_j=0$, and consequently $|{\rm
supp}(t^c)|<s$. Therefore, by part (c), we get $t^c\in\ker(\varphi)$.
\end{proof}

If $(\mathcal{L}^\perp)_\mathcal{X}$ is a standard monomial code, then
$\mathcal{L}^\perp$ has a 
unique basis of standard monomials of $S/I$ because
$\mathcal{L}^\perp$ is the standard function space of
$(\mathcal{L}^\perp)_\mathcal{X}$.
The next result
identifies this basis and classifies when $(\mathcal{L}^\perp)_\mathcal{X}$ is a standard monomial code.

\begin{proposition}\label{monomial-classification-dual} Let $\mathcal{L}$ be a subspace
with a basis of standard monomials $\mathcal{A}=\{t^{a_1},\dots,t^{a_k}\}$. Then
$(\mathcal{L}^\perp)_\mathcal{X}$ is a standard monomial code on $\mathcal{X}$ if and only if
$\mathcal{L}^\perp=K(\Delta_\prec(I)\setminus\mathcal{B}),$
where $\mathcal{B}=\{t^{b_1},\dots,t^{b_k}\}$ is defined in Eq. (\ref{e:set-b}).
\end{proposition}

\begin{proof} $\Rightarrow$) By Proposition~\ref{dual-properties},
$\dim_K(\mathcal{L}^\perp)=|X|-k$. Assume that
$\mathcal{L}^\perp=K\{t^{\gamma_1},\ldots,t^{\gamma_{m-k}}\}$,
$m=|\mathcal{X}|$, $k=\dim_K(\mathcal{L})$. Take $1\leq\ell\leq m-k$. Then,
$t^{\gamma_\ell}$ is in $(\ker(\varphi)\colon\mathcal{L})\bigcap
K\Delta_\prec(I)$. If $t^{\gamma_\ell}$ is in
$\mathcal{B}=\{t^{b_1},\ldots,t^{b_k}\}$, then $t^{\gamma_\ell}=t^{b_i}$
for some $i$, and consequently $t^{b_i}\mathcal{L}\subset{\rm
ker}(\varphi)$. Thus,
$t^{\gamma_\ell}t^{a_i}=t^{b_i}t^{a_i}=\prod_{j=1}^st_j^{d_j}$ and, by
Lemma~\ref{jul6-20-affine}(d), $\prod_{j=1}^st_j^{d_j}$ is not in 
${\rm ker}(\varphi)$, a contradiction. Hence,
$t^{\gamma_\ell}\in\Delta_\prec(I)\setminus\mathcal{B}$, and we obtain the
equality $\{t^{\gamma_1},\ldots,t^{\gamma_{m-k}}\}=\Delta_\prec(I)\setminus\mathcal{B}$ 
because one has the inclusion ``$\subset$'' and these two sets have the same cardinality. 

$\Leftarrow$) This part is clear since
$\Delta_\prec(I)\setminus\mathcal{B}$ consists of standard monomials.
\end{proof}

Following
\cite[p.~16]{Bras-Amoros-O'Sullivan}, we say that a set of monomials
$\mathcal{A}$ 
of $S$ is \textit{divisor-closed} if $t^a\in\mathcal{A}$ whenever $t^a$ divides
a monomial in $\mathcal{A}$.
Families of linear codes generated by monomials that are divisor-closed are also studied
in~\cite{Camps1}. Applications of these sort of codes to polar codes are given in~\cite{Camps2}.
To classify when the algebraic dual of
$\mathcal{L}_\mathcal{X}$ is generated by monomials, we now introduce a weaker notion than
divisor-closed (cf. \cite[Remark~2.5]{Bras-Amoros-O'Sullivan}).

\begin{definition}\label{weakly-divisor-closed-def} A set
$\mathcal{A}=\{t^{a_1},\ldots,t^{a_k}\}\subset\Delta_\prec(I)$,
$t^{a_i}=\prod_{j=1}^st_j^{a_{i,j}}$, of standard monomials
of $S/I$ is \textit{weakly divisor-closed} if 
$$
\frac{t^{a_i}}{\displaystyle\prod_{t_j\in D}t_j^{a_{i,j}}}
$$
is in $\mathcal{A}$ for all monomials $t^{a_i}$ in $\mathcal{A}$ and 
all subsets $D$ of $D_i:=\{t_j\mid a_{i,j}=d_j\}$. 
If $D=\emptyset$, the product $\prod_{t_j\in D}t_j^{a_{i,j}}$ is equal to $1$ by
convention. 
\end{definition}

We come to one of the main results of this section.

\begin{theorem}\label{dual-affine-degenerate}
Let $K$ be a field of characteristic $p$ and let $I$ be the vanishing
ideal of $\mathcal{X}$. Assume that $\gcd(p,e_i)=p$, $e_i=|B_i|$, 
for $i=1,\ldots,s$. Let $\mathcal{L}$ have a monomial basis $\mathcal{A}=\{t^{a_1},\ldots,t^{a_k}\}\subset\Delta_{\prec}(I)$. 
The following are equivalent.
\begin{enumerate}
\item[(a)] $\mathcal{A}$ is weakly divisor-closed. 
\item[(b)] $\mathcal{L}^\perp=K(\Delta_\prec(I)\setminus\mathcal{B})$, where 
$\mathcal{B}=\{t^{b_1},\ldots,t^{b_k}\}$ is defined in Eq. (\ref{e:set-b}).
\item[(c)] $(\mathcal{L}_\mathcal{X})^\perp$ is
a standard monomial code on $\mathcal{X}$. 
\end{enumerate}
\end{theorem}

\begin{proof} (a) $\Rightarrow$ (b): By Lemma~\ref{sep11-20}, one has
the equality $|\Delta_\prec(I)|=|\mathcal{X}|$. Then, by
Proposition~\ref{dual-properties}, one obtains the equalities 
$$
\dim_K(\mathcal{L}^\perp)=|\mathcal{X}|-\dim_K(\mathcal{L})=
|\mathcal{X}|-k=|\Delta_\prec(I)\setminus\mathcal{B}|.
$$
Thus, to show the equality
$\mathcal{L}^\perp=K(\Delta_\prec(I)\setminus\mathcal{B})$, we  only need to show that
$\Delta_\prec(I)\setminus\mathcal{B}\subset\mathcal{L}^\perp$. Take
$t^\alpha\in\Delta_\prec(I)\setminus\mathcal{B}$,
$\alpha=(\alpha_1,\ldots,\alpha_s)$. Since $\mathcal{L}$ is
generated by the set $\{t^{a_1},\ldots,t^{a_k}\}$ it suffices to show that
$t^\alpha t^{a_i}$ is in ${\rm ker}(\varphi)$ for $i=1,\ldots,k$. Fix
$1\leq i\leq k$. 
If
$\alpha_j+a_{i,j}\not\equiv 0\ {\rm mod}(d_j)$ for some $j$, by 
Lemma~\ref{jul6-20-affine}(a), one has $t^\alpha t^{a_i}\in{\rm
ker}(\varphi)$. Thus, we may assume $\alpha_j+a_{i,j}\equiv 0\
{\rm mod}(d_j)$ for $j=1,\ldots,s$. There are
$\lambda_1,\ldots,\lambda_s$ in $\mathbb{N}$ such that
$\alpha_j+a_{i,j}=\lambda_jd_j$ for $j=1,\ldots,s$. By
Lemma~\ref{jul6-20-affine}(c), we may also assume that 
${\rm supp}(t^\alpha t^{a_i})=\{t_1,\ldots,t_s\}$ 
and $\lambda_j\geq 1$ for all $j$. If $\lambda_j=1$ for all $j$, we
obtain that $t^\alpha=t^{b_i}$, a contradiction. If $\lambda_j\geq 2$
for some $j$, since $\alpha_j+a_{i,j}=\lambda_jd_j\leq
2d_j$, we obtain that $\lambda_j=2$ and $\alpha_j=a_{i,j}=d_j$.
Therefore, for each $1\leq j\leq s$ either  
$\alpha_j+a_{i,j}=d_j$ or $\alpha_j=a_{i,j}=d_j$. Next we show that
this cannot occur. We set 
$$
t^\delta:=\frac{t^{a_i}}{\displaystyle\prod_{t_j\in D}t_j^{a_{i,j}}},
$$
where $D:=\{t_j\mid\alpha_j=a_{i,j}=d_j\}$ is a subset of
$D_i=\{t_j\mid a_{i,j}=d_j\}$. Since the set $\mathcal{A}$ is 
weakly divisor-closed, we get $t^\delta\in\mathcal{A}$, and $t^\delta=t^{a_r}$
for some $1\leq r\leq k$. From the equalities $t^\alpha
t^{a_r}=t^\alpha t^\delta=\prod_{j=1}^st_j^{d_j}$, we obtain 
$t^{\alpha}\in\mathcal{B}$, a contradiction.

(b) $\Rightarrow$ (a): Assume that
$\mathcal{L}^\perp=K(\Delta_\prec(I)\setminus\mathcal{B})$. Take $t^{a_i}$
in $\mathcal{A}$ and let $D$ be a subset of $D_i$. If $D=\emptyset$,
there is nothing to prove. For simplicity
of notation we may assume that $D=\{t_1,\ldots,t_\ell\}$ for some
$1\leq\ell\leq s$. 
Then, we can write $t^{a_i}=t_1^{d_1}\cdots
t_\ell^{d_\ell}t^\gamma$, where 
$t^\gamma=t_{\ell+1}^{a_{i,\ell+1}}\cdots t_{s}^{a_{i,s}}$ and $1\leq
a_{i,j}\leq d_j$ for $j>\ell$. If $\ell=s$, 
by convention $t^\gamma=1$. To show that $\mathcal{A}$ is weakly
divisor-closed we need only show $t^\gamma\in\mathcal{A}$. We
proceed by contradiction assuming $t^\gamma\notin\mathcal{A}$.
From the equality 
$$
(t_{1}^{d_{1}}\cdots
t_{\ell}^{d_{\ell}}t_{\ell+1}^{d_{\ell+1}-a_{i,\ell+1}}\cdots
t_{s}^{d_s-a_{i,s}})t^\gamma=t_1^{d_1}\cdots t_s^{d_s},
$$
we obtain that the monomial $t^u:=t_{1}^{d_{1}}\cdots
t_{\ell}^{d_{\ell}}t_{\ell+1}^{d_{\ell+1}-a_{i,\ell+1}}\cdots
t_{s}^{d_s-a_{i,s}}$ is not in $\mathcal{B}$ since $t^\gamma$ is not in 
$\mathcal{A}$. Thus,
$t^u\in\Delta_\prec(I)\setminus\Gamma\subset\mathcal{L}^\perp$. Then,
$t^u\mathcal{L}\subset{\rm ker}(\varphi)$. As $t^{a_i}$ is in
$\mathcal{L}$, one has
\begin{equation}\label{jul20-20}
t_1^{2d_1}\cdots
t_{\ell}^{2d_\ell}t_{\ell+1}^{d_{\ell+1}}\cdots t_s^{d_s}=t^ut^{a_i}\in{\rm
ker}(\varphi),
\end{equation}
a contradiction because by Lemma~\ref{jul6-20-affine}(d), the left hand side of
Eq.~(\ref{jul20-20}) is not in ${\rm ker}(\varphi)$. 

(b) $\Leftrightarrow$ (c): The linear code $(\mathcal{L}_\mathcal{X})^\perp$ is the
evaluation code $(\mathcal{L}^\perp)_\mathcal{X}$ on $\mathcal{X}$ by
Theorem~\ref{formula-dual}. Thus, that (b) and (c) are equivalent 
follows from Proposition~\ref{monomial-classification-dual}.  
\end{proof}

\begin{corollary}\label{direct-application-clas} If
$\mathcal{X}=K$,
$\mathcal{L}=K\{t_1^{a_1},\ldots,t_1^{a_k}\}$, and $0\leq a_1<\cdots<a_k\leq
q-1$, then $(\mathcal{L}_\mathcal{X})^\perp$ is 
a standard monomial code on $\mathcal{X}$ if and only if either $a_k<q-1$ or $a_k=q-1$ and $1\in\mathcal{L}$.
\end{corollary}

\begin{proof} We set $A_1=K^*$, $B_1=K$,
$d_1=q-1$, and $e_1=q$. Note that $\gcd(q,p)=p$, where $p={\rm
char}(K)$. By Theorem~\ref{dual-affine-degenerate}, it suffices to note that
$\mathcal{A}=\{t_1^{a_1},\ldots,t_1^{a_k}\}$ is weakly divisor-closed if and
only if either $a_k<q-1=d_1$ or $a_k=q-1=d_1$ and $1\in\mathcal{L}$.
\end{proof}
 
\begin{corollary}\cite[Proposition~2.4]{Bras-Amoros-O'Sullivan}\label{dual-affine-coro} 
Let $\mathcal{L}_\mathcal{X}$ be a standard monomial code on
$\mathcal{X}=K^s$, let 
$\mathcal{A}$ be the monomial basis of
$\mathcal{L}$, and let $I$ be the vanishing ideal of
$\mathcal{X}$. If
$\mathcal{A}$ is divisor-closed, then  
$$\mathcal{L}^\perp=\Delta_\prec(I)\setminus\{t_1^{q-1-c_1}\cdots
t_s^{q-1-c_s} :\, t_1^{c_1}\cdots t_s^{c_s}\in\mathcal{A}\}\
\mbox{ and }\ (\mathcal{L}_\mathcal{X})^\perp=(\mathcal{L}^\perp)_\mathcal{X}. 
$$ 
\end{corollary}
\begin{proof} It follows from Theorems~\ref{formula-dual} and 
\ref{dual-affine-degenerate} by making
$B_i=K$ for $i=1,\ldots,s$.
\end{proof}

We now determine the algebraic dual of $K(S_{\leq
d}\bigcap \Delta_\prec(I(\mathcal{X}))$.

\begin{theorem}\label{algebraic-dual-reed-muller} 
Let $K$ be a field of characteristic $p$ and let $I$ be the
vanishing ideal of $\mathcal{X}$. Assume that $\gcd(e_i,p)=p$,
$e_i=|B_i|$, for all $i$. 
If $1\leq
d<r_0=\sum_{i=1}^s(e_i-1)$, $\mathcal{A}:=S_{\leq
d}\bigcap\Delta_\prec(I)=\{t^{a_1},\ldots,t^{a_k}\}$, 
and $\mathcal{L}=K\mathcal{A}$, then 
$$\mathcal{L}^\perp=K(\Delta_\prec(I)\setminus\mathcal{B})=K(S_{\leq
r_0-d-1}\textstyle\bigcap\Delta_\prec(I)),
$$
where $\mathcal{B}=\{t^{b_1},\ldots,t^{b_k}\}$ is defined in Eq. (\ref{e:set-b}).
\end{theorem}

\begin{proof} We set $\mathcal{A}^\perp:=S_{\leq
r_0-d-1}\bigcap\Delta_\prec(I)$,   
and $k=\dim_K(K\mathcal{A})=H_I^a(d)$. 
Note that $|\mathcal{A}|=|\mathcal{B}|$. We claim that
$\mathcal{A}^\perp=\Delta_\prec(I)\setminus\mathcal{B}$. As $I=I(\mathcal{X})$ is a complete
intersection, by Corollary~\ref{duality-hilbert-gorenstein}, 
one has $H_I^a(d)+H_I^a(r_0-d-1)=|\mathcal{X}|$.
Then, using Lemmas~\ref{lemma-referee1} and \ref{sep11-20}, one has
$$
H_I^a(r_0-d-1)=|\Delta_\prec(I)\textstyle\bigcap
S_{\leq r_0-d-1}|=|\mathcal{A}^\perp|=|\mathcal{X}|-H_I^a(d)=|\Delta_\prec(I)\setminus\mathcal{B}|.
$$
\quad Hence, to show that $\mathcal{A}^\perp=\Delta_\prec(I)\setminus\mathcal{B}$, we need only show
the inclusion ``$\supset$''. Given $c\in\mathbb{N}^n$,
$c=(c_1,\ldots,c_s)$, we set $|c|:=\sum_{j=1}^sc_j$. Take
$t^c\in\Delta_\prec(I)\setminus\mathcal{B}$, $c=(c_1,\ldots,c_s)$. We
proceed by contradiction. Assume that $t^c\notin\mathcal{A}^\perp$,
that is, $|c|>r_0-d-1$. Setting $a=(d_1-c_1,\ldots,d_s-c_s)$, we get 
$|a|=r_0-|c|<d+1$. Thus, $|a|\leq d$ and $t^a\in\mathcal{A}$. As
$t^ct^a=\prod_{j=1}^st_j^{d_j}$, we get $t^c\in\mathcal{B}$, a contradiction.
This proves the claim. Hence,
$K\mathcal{A}^\perp=K(\Delta_\prec(I)\setminus\mathcal{B})$. The equality 
$\mathcal{L}^\perp=K(\Delta_\prec(I)\setminus\mathcal{B})$
follows from Theorem~\ref{dual-affine-degenerate} because
$\mathcal{A}$ is weakly divisor-closed.
\end{proof}

If $\mathcal{L}=S_{\leq d}$, the evaluation code $\mathcal{L}_\mathcal{X}$ on
$\mathcal{X}$, denoted by $C_\mathcal{X}(d)$, is the Reed--Muller-type code
on $\mathcal{X}$ of  degree $d$. 
 The codes
$C_\mathcal{X}(d)^\perp$ and $C_\mathcal{X}(r_0-d-1)$ are equivalent
(Corollary~\ref{gorenstein-codes}, \cite[Theorem~5.7]{GHWCartesian},
\cite[Theorem~2.3]{Lopez-Manganiello-Matthews}), the next result
shows that they are equal when $\mathcal{X}$ is a degenerate affine space and
${\rm char}(K)$ divides $e_i$ for all $i$. 

\begin{proposition}\label{dual-affine-degenerate-reed-muller}
Let $K$ be a field of characteristic of $p$ such
that $\gcd(e_i,p)=p$, $e_i=|B_i|$, for all $i$. Then,
$C_\mathcal{X}(d)^\perp=C_\mathcal{X}(r_0-d-1)$ if $d<r_0$ and 
$C_\mathcal{X}(d)^\perp=(0)$ if $d=r_0$.
\end{proposition}

\begin{proof} Let $\prec$ be a graded monomial order. 
If $\mathcal{L}=S_{\leq d}$ and $I=I(\mathcal{X})$, then the standard function space
$\widetilde{\mathcal{L}}$ of $\mathcal{L}_X$ is $K(S_{\leq
d}\bigcap \Delta_\prec(I))$ and 
$\widetilde{\mathcal{L}}_\mathcal{X}=\mathcal{L}_\mathcal{X}=C_\mathcal{X}(d)$
(see Proposition~\ref{transforming-new}, Corollary~\ref{unique-standard}). If $d<r_0$, then by
Theorem~\ref{algebraic-dual-reed-muller},
we obtain the equality $(\widetilde{\mathcal{L}})^\perp=K(S_{\leq
r_0-d-1}\bigcap\Delta_\prec(I))$. Therefore, using
Theorem~\ref{formula-dual}, one has
\begin{eqnarray*}
C_\mathcal{X}(d)^\perp&=&(\widetilde{\mathcal{L}}_\mathcal{X})^\perp=
((\widetilde{\mathcal{L}})^\perp)_\mathcal{X}=(K(S_{\leq
r_0-d-1}\textstyle\bigcap\Delta_\prec(I)))_\mathcal{X}\\
&=&(S_{\leq r_0-d-1})_\mathcal{X}=C_\mathcal{X}(d-r_0-1).
\end{eqnarray*}
\quad If $d=r_0$, then $C_\mathcal{X}(d)=K^{m}$, $m=|\mathcal{X}|$, and $C_\mathcal{X}(d)^\perp=(0)$.
\end{proof}

\section{Examples}\label{examples-section}
This section includes examples illustrating some of our
results. In Appendix~\ref{Appendix} we give the implementations 
in \textit{Macaulay}$2$ \cite{mac2} that are used in some of the
examples. The monomial order $\prec$ that we use in the following
examples is the graded reverse lexicographical order (GRevLex order)
\cite[p.~343]{monalg-rev}. This is the default order 
in \textit{Macaulay}$2$.

\begin{example}\label{8points-in-A3}
Let $K$ be the finite field $\mathbb{F}_3$, let $S=K[t_1,t_2,t_3]$ be
a polynomial ring, 
let $\prec$ be the GRevLex
order on $S$, let $I=I(X)$ be the vanishing ideal of the set of
evaluation points
$$
X=\{(1,1,1),\, (1,1,-1),\, (0,0,0),\, (0,0,1),\, (0,0,-1),\,
(0,1,0),\, (0,1,1),\, (0,1,-1)\},
$$
let $P_i$ be the point in $X$ in the $i$-th position from the left, and let $\mathfrak{p}_i$ be the
vanishing ideal of $P_i$. The ideal $I$ is generated by 
$$\mathcal{G}=\{t_2^2-t_2, t_1t_2-t_1,
t_1^2-t_1,t_3^3-t_3, t_1t_3^2-t_1\}
$$
and this set is a Gr\"obner basis
for $I$. The Reed--Muller code $C_X(2)$ is the standard evaluation
code $\mathcal{L}_X$ where $\mathcal{L}$ is the standard function
space of $C_X(2)$ spanned 
by the set of remainders 
$$\{1,\,t_3,\,t_2,\,t_1,\,t_3^2,\,t_2t_3,\,t_1t_3\}$$
of the monomial basis of $S_{\leq 2}$ on
division by $\mathcal{G}$ (Proposition~\ref{transforming-new}). 
The algebraic dual of $C_X(2)$ is $\mathcal{L}^\perp=K(t_1+t_2+1)$ and
the dual of $C_X(2)$ is $K(0,0,1,1,1,1,1,1)$.

The homogenization $I^h$ of the ideal $I$ is not Gorenstein, the rings $S/I^h$ and $S/{\rm
in}_\prec(I)$ have symmetric $h$-vector given by $(1,3,3,1)$, and
$r_0={\rm reg}(H_I^a)=3$. The Reed--Muller code $C_X(1)$ is the
standard evaluation code $\mathcal{L}_X$, where 
$\mathcal{L}=S_{\leq 1}$. The algebraic dual of 
$C_X(1)$ is  
$$
\mathcal{L}^\perp=K\{t_1t_3 + t_2t_3 - t_1 - t_2 - t_3 - 1,\,  t_1 + t_2
+ t_3 + 1,\,  t_2 + t_3 - 1,\,  t_3\}. 
$$
\quad If $d=1$, the linear code $C_X(d)$ is not monomially equivalent to
$C_X(r_0-d-1)^\perp$ because their minimum distances are $\delta(C_X(1))=2$ and
$\delta(C_X(1)^\perp)=3$, respectively. Hence, the duality
criterion of Theorem~\ref{duality-criterion} fails if we replace
condition (b) by $H_I^a(d)+H_I^a(r_0-d-1)=|X|$ for $-1\leq d\leq r_0$.
Condition (b) of Theorem~\ref{duality-criterion} is not satisfied
because the unique list, up to multiplication by scalars from
$K^*$, of standard indicator function for $X$ is: 
\begin{align*}
&f_1=t_1t_3+t_1 ,\,  f_2=t_1t_3-t_1 ,\,  f_3=t_2t_3^2-t_3^2-t_2+1 ,\,
f_4=t_2t_3^2+t_2t_3-t_3^2-t_3,\\
&f_5=t_2t_3^2-t_2t_3-t_3^2+t_3,\,
f_6=t_2t_3^2-t_2 ,\,  f_7=t_2t_3^2-t_1t_3+t_2t_3-t_1,\\
&f_8=t_2t_3^2+t_1t_3-t_2t_3-t_1,
\end{align*}
and ${\rm v}_{\mathfrak{p}_i}(I)=\deg(f_i)$ for all $i$
(Proposition~\ref{indicator-function-prop}(a)).  
In particular one has ${\rm v}_{\mathfrak{p}_1}(I)=2$, ${\rm
v}_{\mathfrak{p}_3}(I)=3$, ${\rm v}(I)=2$ and, by
Proposition~\ref{tuesday-afternoon}, the index of regularity ${\rm
reg}(\delta_X)$ of $\delta_X$ is $2$. This example corresponds to
Procedure~\ref{8points-in-A3-procedure}. 
\end{example}

\begin{example}\label{Ivan-Hiram}
Let $K$ be the finite field $\mathbb{F}_3$, let $S=K[t_1,t_2,t_3]$ be
a polynomial ring, 
let $\prec$ be the GRevLex
order on $S$, let $I=I(X)$ be the vanishing ideal of the set of
evaluation points
$$
X=\{(1,0,0),\, (0,1,0),\, (0,0,1),\, (0,0,0),\, (2,2,2)\},
$$
let $P_i$ be the point in $X$ in the $i$-th position from the left, and let $\mathfrak{p}_i$ be the
vanishing ideal of $P_i$. Adapting
Procedure~\ref{8points-in-A3-procedure}, we obtain the following data.
The ideal $I$ is generated by 
\begin{align*}
\mathcal{G}=&\{t_2t_3+t_3^2-t_3,\, t_1t_3+t_3^2-t_3,\,
t_2^2-t_3^2-t_2+t_3,\, \\ 
&\ t_1t_2+t_3^2-t_3,\, t_1^2-t_3^2-t_1+t_3,\, 
t_3^3-t_3\},
\end{align*}
and this set is a Gr\"obner basis for $I$. 
Then, $\mathcal{G}^h=\{g^h\mid g\in \mathcal{G}\}$ is a Gr\"obner basis
for $I^h$, the homogenization of $I$ with respect to $u$
\cite[Proposition~3.4.2]{monalg-rev}. The ideal $I^h$ is Gorenstein
because  $I^h$ is a Cohen--Macaulay ideal of height $3$ and 
the minimal resolution of $S[u]/I^h$ by free $R$-modules, 
$R=S[u]$, is
given by 
$$
0\longrightarrow R(-5)\longrightarrow R(-3)^5\longrightarrow
R(-2)^5\longrightarrow R\longrightarrow R/I^h\longrightarrow 0, 
$$
see \cite[Corollary~5.3.5]{monalg-rev}. It is seen that $I$ is not a
complete intersection, that is, $I$ cannot be generated by $3$
elements. The graded rings $S/I^h$ and $S/{\rm in}_\prec(I)$ have
symmetric $h$-vector given by 
$(1,3,1)$ and $r_0={\rm reg}(H_I^a)=2$. The sorted list of standard
monomials of $S/I$ is
$$
\Delta_\prec(I)=\{1,t_3,t_2,t_1,t_3^2\},
$$
and the unique set $F=\{f_i\}_{i=1}^5$ of standard indicator functions
for $X$ with $f_i(P_i)=1$ for all $i$ is
\begin{align*}
&f_1=t_3+t_1-t_3^2,\, f_2=t_3+t_2-t_3^2,\, f_3=-t_3-t_3^2,\\ 
&f_4=1+t_3-t_2-t_1+t_3^2,\, f_5=t_3-t_3^2.
\end{align*}
\quad Setting $g=-f_1-f_2-f_3+f_4-f_5$, one has $g(P_i)=-1$ for $i\neq
4$ and $g(P_4)=1$. By Theorem~\ref{duality-criterion}  we obtain 
$$
C_X(1)^\perp=(g(P_1),\ldots,g(P_5))\cdot C_X(0)=K(-1,-1,-1,1,-1).
$$
\end{example}

\begin{example}\label{7points-in-A3} 
Let $K$ be the finite field $\mathbb{F}_3$, let $S=K[t_1,t_2,t_3]$ be
a polynomial ring,
let $\prec$ be the GRevLex
order on $S$, let $I=I(X)$ be the vanishing ideal of the set of
evaluation points
$$
X=\{(1,1,-1),\, (0,0,0),\, (0,0,1),\, (0,0,-1),\,
(0,1,0),\, (0,1,1),\, (0,1,-1)\},
$$
let $P_i$ be the point in $X$ in the $i$-th position from the left, 
and let $\mathfrak{p}_i$ be the
vanishing ideal of $P_i$. Adapting 
Procedure~\ref{8points-in-A3-procedure}, we obtain the following 
data. The Reed--Muller code $C_X(1)$ is the
standard evaluation code $\mathcal{L}_X$, where 
$\mathcal{L}=S_{\leq 1}$. The algebraic dual of 
$C_X(1)$ is  
$$
\mathcal{L}^\perp=K\{t_2t_3 - t_1 - t_2 + t_3 + 1,\,  t_1 + t_2 +
1,\,  t_2 - 1\}, 
$$
the minimum distances of $C_X(1)$ and its dual $C_X(1)^\perp$ are $\delta(C_X(1))=1$
and $\delta(C_X(1)^\perp)=3$, and $r_0={\rm
reg}(H_I^a)=3$. 
If $d=1$, then 
$$H_I^a(d)+H_I^a(r_0-d-1)=H_I^a(1)+H_I^a(1)=8>7=|X|.$$ 
\quad Hence, the inequality of Proposition~\ref{inequality-dim-dual}(a) does not hold
in general. The local v-numbers are ${\rm v}_{\mathfrak{p}_1}(I)=1$
and ${\rm v}_{\mathfrak{p}_i}(I)=3$
for $i\geq 2$. 
In particular ${\rm v}(I)=1$ and, by
Proposition~\ref{tuesday-afternoon}, the index of regularity ${\rm
reg}(\delta_X)$ of $\delta_X$ is $1$. If $d=2$, then 
$$H_I^a(d)+H_I^a(r_0-d-1)=H_I^a(2)+H_I^a(0)=6+1=|X|,$$ 
the algebraic dual of $C_X(2)$ is $K\{t_1+t_2+1\}$, and
$\delta(C_X(2)^\perp)=6$. The algebraic dual of $C_X(0)$ is 
equal to $\ker(\varphi)\bigcap K\Delta_\prec(I)$ and is given by 
\begin{align*}                                                                         
&K\{t_2t_3^2  - t_2t_3 - t_3^2  - t_1 - t_2 - t_3 - 1,\, t_2t_3
+t_3^2  + t_1
+ t_2 + t_3 + 1,\\
& \ \ \ \  t_3^2  +t_1 + t_2 + t_3,\,  t_1 + t_2 + t_3 - 1,\, 
t_2 + t_3,\ t_3 +1\},
\end{align*}
$\delta(C_X(0)^\perp)=2$ and $\delta(C_X(2))=1$. For $d=0$,
$C_X(d)^\perp$ is not equivalent to $C_X(r_0-d-1)$. This example proves that 
in Theorem~\ref{duality-criterion}(b) and
Corollary~\ref{geramita-gorenstein} the assumption ``$r_0={\rm
v}_{\mathfrak{p}}(I)$ for  
$\mathfrak{p}\in{\rm Ass}(I)$'' is essential.
\end{example}

\begin{example}\label{one-variable} Let $X=\{P_1,\ldots,P_m\}$ be a subset of
$K=\mathbb{F}_q$, $m\geq 2$, let $S=K[t_1]$ be a
polynomial ring in one variable, and let $F=\{f_1,\ldots,f_m\}$ be the unique set of
standard indicator functions for
$X$ such that $f_i(P_i)=1$ for all $i$. Then
the vanishing ideal $I=I(X)$ is a principal ideal generated by
$\prod_{i=1}^m(t_1-P_i)$, $r_0={\rm reg}(H_I^a)=m-1$,   
$$f_i=\prod_{j\neq i}(t_1-P_j)\bigg/ \prod_{j\neq i}(P_i-P_j),$$
and ${\rm lc}(f_i)=\left[\prod_{j\neq i}(P_i-P_j)\right]^{-1}$ for
$i=1,\ldots,m$. By Theorem~\ref{duality-criterion}  we obtain 
$$
C_X(r_0-d-1)^\perp=(g(P_1),\ldots,g(P_m))\cdot C_X(d)\ \text{ for\ \
$-1\leq d\leq
r_0$},
$$
where $g={\rm lc}(f_1)f_1+\cdots+{\rm lc}(f_m)f_m$ and $g(P_i)={\rm lc}(f_i)$
for all $i$. If $X=K$, then 
$$
K^*=\{P_i-P_1,\ldots,P_i-P_{i-1},P_i-P_{i+1},\ldots,P_i-P_m\},
$$
${\rm lc}(f_i)={\rm lc}(f_m)$ for $i=1,\ldots,m$, and
$C_X(r_0-d-1)^\perp=C_X(d)$ for $-1\leq d\leq r_0$. 
\end{example}

\begin{example}\label{Hiram-example} 
Let $S=K[t_1]$ be a
polynomial ring in one variable over the field $K=\mathbb{F}_7$, let $\beta$ be a 
generator of $K^*$, and let
$X$ be the set of points $=\{\beta^6,\beta,\beta^4,\beta^5\}=\{1,3,4,5\}$.  The vanishing ideal
$I$ of $X$ is generated by
$(t_1-\beta^6)(t_1-\beta)(t_1-\beta^4)(t_1-\beta^5)$ and 
$r_0={\rm reg}(H_I^a)=3$. Let $\prec$ be the GRevLex order. 
The set of standard monomials of $S/I$ is 
$$
\Delta_\prec(I)=\{1,t_1,t_1^2,t_1^3\}.
$$
\quad If $\mathcal{L}=K\{1,t_1,t_1^2\}$, then $\mathcal{L}_X=C_X(2)$.
Adapting Procedure~\ref{8points-in-A3-procedure}, we obtain that the
algebraic dual $\mathcal{L}^\perp$ of $\mathcal{L}$ is $Kg$,  
where $g$ is the polynomial $t_1^3-t_1^2-2t_1$. Evaluating $g$ at each
point of $X$ gives the vector $(-2,-2,-2,-1)$ and 
$$
C_X(2)^\perp=K(2,2,2,1).
$$
\quad The unique set, up to multiplication by scalars from
$K^*$, of standard indicator functions for the points $\beta^6,\beta,\beta^4,\beta^5$ are 
\begin{align*}
&f_1=t_1^3+2t_1^2-2t_1+3,\,
f_2=t_1^3-3t_1^2+t_1+1,\,f_3=t_1^3-2t_1^2+2t_1-1,\,\\
&f_4=t_1^3-t_1^2-2t_1+2,
\end{align*}
respectively, and they generate $K\Delta_\prec(I)$
(Proposition~\ref{indicator-function-prop}(a)). 
The v-number of 
$I$ at each point of $X$ is $3$, ${\rm
v}(I)=3$, and
$C_X(r_0-d-1)^\perp$ is equivalent to $C_X(d)$ for $-1\leq d\leq r_0$
(Corollary~\ref{gorenstein-codes}). 
\end{example} 

\begin{example}\label{5points-in-A2}
Let $S=K[t_1,t_2]$ be a polynomial ring over the field
$K=\mathbb{F}_3$, let $X$ be the set  
$$X=\{(0,0),\,(1,0),\,(0,1),\,(1,1),\,(0,-1)\},$$ 
let $I=I(X)$ be the vanishing ideal of
$X$, and let $\prec$ be the GRevLex order on $S$. Adapting 
Procedure~\ref{8points-in-A3-procedure}, we obtain the following 
data. The ideal
$I$ is generated by 
$$\mathcal{G}=\{t_1^2-t_1,\, t_2^3-t_2,\,
t_1t_2^2-t_1t_2\},
$$
and this set is a Gr\"obner basis of $I$. Let
$\mathcal{L}$ be the monomial space $K\{1,t_1,t_2\}$ of $S$. Then one
has 
$\mathcal{L}\simeq C_X(1)$, $\mathcal{L}_X=C_X(1)$, 
$$
\mathcal{L}^\perp=K\{t_1t_2-t_1+t_2,\, t_1-1\},\ \
C_X(1)^\perp=K\{(1,0,1,0,1),(0,1,-1,-1,1)\},
$$
$\delta(C_X(1))=2$, $\delta(C_X(1)^\perp)=3$, ${\rm
v}_{\mathfrak{p}}(I)=2$ for $\mathfrak{p}\in{\rm Ass}(I)$,
$H_I^a(1)=\dim_K(C_X(1))=3$, 
$H_I^a(2)=\dim_K(C_X(2))=5$, and $r_0={\rm reg}(H^a_{I})=2$. The
unique set, up to multiplication by scalars from
$K^*$, of standard indicator functions for the points of $X$ of
Proposition~\ref{indicator-function-prop}(a) are given by 
$$
f_1=t_1t_2-t_2^2-t_1+1,\, f_2=t_1t_2-t_1,\, f_3=t_1t_2+t_2^2+t_2,\,
f_4=t_1t_2,\,
f_5=t_2^2-t_2,
$$
and $4=H_I^a(d)+H_I^a(r_0-d-1)<|X|=5$ for $d=1$.
\end{example}

\begin{example}\label{example1} 
Let $S=K[t_1,t_2]$ be a polynomial ring over the field
$K=\mathbb{F}_7$, let $\beta$ be a generator of the cyclic group
$K^*$, and let $A_i$, $i=1,2$, be the cyclic groups $A_1=(\beta^2)$,
$A_2=(\beta^3)$. The orders of $\beta^2$ and $\beta^3$ are $d_1=3$ and
$d_1=2$, respectively. Let $\mathcal{L}$ be the linear space 
generated by $\mathcal{B}=\{1,t_1,t_2,t_1t_2\}$ and let 
$\mathcal{L}_T$ be the monomial standard evaluation 
code on $T=A_1\times A_2$ relative to the GRevLex order $\prec$. The
vanishing ideal $I=I(T)$ of $T$ is generated by $t_1^3-1$ and
$t_2^2-1$, the index of regularity of $H_I^a$ is $3$, and the set of
standard monomials of $S/I$ is 
$$
\Delta_\prec(I)=\{1,t_1,t_2,t_1^2,t_1t_2,t_1^2t_2\}.
$$
\quad According to Proposition~\ref{dual-toric-degenerate} and
Corollary~\ref{dual-toric-degenerate-coro}, the
algebraic dual $\mathcal{L}^\perp$ is given by 
$$
\mathcal{L}^\perp=K(\Delta_\prec(I)\setminus\mathcal{B})=K\{t_1,t_1t_2\},
\quad  \text{where }\ \mathcal{B}=\{1,t_1^2,t_2,t_1^2t_2\},
$$
and $(\mathcal{L}_T)^\perp=(\mathcal{L}^\perp)_T$. The minimum
distance $\delta(\mathcal{L}_T)$ of $\mathcal{L}_T$
is $2$ and $\delta((\mathcal{L}_T)^\perp)=3$. 
\end{example}

\begin{example}\label{example2} 
Let $S=K[t_1,t_2]$ be a polynomial ring over the field
$K=\mathbb{F}_4$. We set $\mathcal{X}=K^2$, $d_1=d_2=3$, $e_1=e_2=4$, and  
$$ 
t^{a_1}=1,\, t^{a_2}=t_1,\, t^{a_3}=t_2,\, t^{a_4}=t_2^2,\,
t^{a_5}=t_2^3,\, t^{a_6}=t_1t_2^2.
$$
\quad Let $\mathcal{L}$ be the linear subspace 
of $S$ generated by $\mathcal{A}=\{t^{a_1},\ldots,t^{a_6}\}$ and let 
$\mathcal{L}_\mathcal{X}$ be the monomial standard evaluation 
code on $\mathcal{X}$ relative the GRevLex order $\prec$. The
vanishing ideal $I=I(\mathcal{X})$ of $\mathcal{X}$ is generated by $t_1^{e_1}-t_1$ and
$t_2^{e_2}-t_2$, the index of regularity of $H_I^a$ is $r_0=6$, the set of
standard monomials of $S/I$ is 
$$
\Delta_\prec(I)=\{1,\,t_1,\,t_2,\,t_2^2,\,t_1t_2,\,t_1^2,\,t_2^3,\,t_1t_2^2,\,t_1^2t_2,\,
t_1^3,\,t_1t_2^3,\,t_1^2t_2^2,\,t_1^3t_2,\,t_1^3t_2^3,\,t_1^2t_2^3,\,t_1^3t_2^2\},
$$
and $\mathcal{A}$ is weakly divisor-closed. Setting $t^{b_1}=t_1^3t_2^3$, $t^{b_2}=t_1^2t_2^3$,
$t^{b_3}=t_1^3t_2^2$, $t^{b_4}=t_1^3t_2$, $t^{b_5}=t_1^3$,
$t^{b_6}=t_1^2t_2$, according to Theorem~\ref{dual-affine-degenerate}, the
algebraic dual $\mathcal{L}^\perp$ is given by 
$$
\mathcal{L}^\perp=K(\Delta_\prec(I)\setminus\mathcal{B})=K\{1,\,t_1,\,t_2,\,t_2^2,\,t_1t_2,
\,t_1^2,\,t_2^3,\,t_1t_2^2,\,t_1t_2^3,\,t_1^2t_2^2\},
$$
where $\mathcal{B}=\{t^{b_1},\ldots,t^{b_6}\}$, and
$(\mathcal{L}_\mathcal{X})^\perp=(\mathcal{L}^\perp)_\mathcal{X}$. The minimum
distance $\delta(\mathcal{L}_\mathcal{X})$ of $\mathcal{L}_\mathcal{X}$
is $4$. Other examples of sets that are weakly divisor-closed are
$\mathcal{A}\cup\{t_1^3t_2^3,t_1^3\}$ and
$\mathcal{A}\cup\{t_1^3t_2^3,t_1^3, t_1^2t_2^3, t_1^2\}$.
\end{example}

\begin{appendix}

\section{Procedures for {\it Macaulay\/}$2$}\label{Appendix}

In this appendix we give a procedure for \textit{Macaulay}$2$
\cite{mac2} that is used in some of the examples presented in
Section~\ref{examples-section}. We use the package NAGtypes, written
by Anton Leykin, that defines types used by the
package NumericalAlgebraicGeometry as well as other numerical
algebraic geometry packages.  

\begin{procedure}\label{8points-in-A3-procedure}
Let $\mathcal{L}_X$ be an evaluation code and let $ \prec$ be the graded reverse
lexicographical order (GRevLex order) \cite[p.~343]{monalg-rev}, 
which is the default order in \textit{Macaulay}$2$ \cite{mac2}. This procedure computes the
standard function space $\widetilde{\mathcal{L}}$ and the minimum distance of $\mathcal{L}_X$. It
determines whether or not 
the algebraic dual $\mathcal{L}^\perp$ of $\mathcal{L}_X$ is
generated by monomials. If not, it computes a generating set for
$\mathcal{L}^\perp$ and then, using the algorithm of
Theorem~\ref{dim-algo}, it
computes a $K$-basis for $\mathcal{L}^\perp$. This procedure also
computes the vanishing ideal $I=I(X)$, the regularity index $r_0$ of the affine
Hilbert function $H_I^a$, the unique set, up to multiplication by scalars from
$K^*$, of the standard indicator functions for $X$ (Remark~\ref{nov23-20}), and the v-numbers
associated to $I$. This procedure can be used to check the
condition ``$H_I^a(r_0-d-1)+H_I^a(d)=|X|$ for $0\leq d\leq r_0$'' using the
$h$-vector of the homogenization of $I$
(Proposition~\ref{duality-hilbert-function}), and to determine whether or
not the homogenization $I^h$ of the ideal $I$ is Gorenstein.
This 
procedure corresponds to Example~\ref{8points-in-A3}.

\begin{verbatim}
load "NAGtypes.m2"
q=3, Fq=GF(q,Variable=>a), S=Fq[t1,t2,t3]
--Evaluation points of the code:
X={{1,1,1},{1,1,-1},{0,0,0},{0,0,1},{0,0,-1}, {0,1,0},
{0,1,1},{0,1,-1}}
--Vanishing ideals of the points:
I1=ideal(t1-1,t2-1,t3-1),I2=ideal(t1-1,t2-1,t3+1),
I3=ideal(t1,t2,t3),I4=ideal(t1,t2,t3-1),I5=ideal(t1,t2,t3+1),
I6=ideal(t1,t2-1,t3),I7=ideal(t1,t2-1,t3-1),I8=ideal(t1,t2-1,t3+1)
I=intersect(I1,I2,I3,I4,I5,I6,I7,I8)--Vanishing ideal
L={I1,I2,I3,I4,I5,I6,I7,I8}--List of ideals
G=gb I, M=coker gens G
r0=regularity I-1 --Regularity of H_I^a
--Computes the remainder of x on division by G:
div=(x)->x % G
--Monomials that define the evaluation code:
Basis=matrix{{1,t1,t2,t3}}
--The list of remainders of Basis after division by G
--gives the standard function space:
cL=toList set apply(flatten entries Basis,div)
(d,r)=(1,1)
--This is the set of all elements of the ground field Fq:
field=set(apply(1..q-1,n->a^n))+set{0}
--Var1 to Var6 are used to compute the minimum distance of L_X:
Var1=(field)^**(#cL)-(set{0})^**(#cL)
Var2=apply(toList (Var1)/deepSplice,toList)
Var3=apply(Var2,x->matrix{cL}*vector x)
Var4=set(apply(apply(Var3,entries),n->n#0))
Var5=subsets(toList set apply(toList Var4,
m->(leadCoefficient(m))^(-1)*m),r)
Var6=apply(apply(Var5,ideal), x-> if #(set flatten entries 
leadTerm gens x)==r then degree(I+x) else 0)
md=degree M-max Var6--Minimum distance
ps=(n)->polySystem(n*cL)
--We are redefining d to be r0 to compute the 
--algebraic dual
(d,r)=(r0, 1)
--These are the points of X in the right format:
B=apply(X,x->point{toList x})
--The number of elements of P1 is the length of the code
P1=apply(flatten entries basis(0,r0,M),div)
b1=apply(P1,ps)
funct1=(n)->apply(B,x->evaluate(b1#n,x))
MatA=matrix{apply(0..#P1-1,n->{a^(q-1)})}
MatB=matrix{apply(0..#cL-1,n->{a^(q-1)-1})}
funct2=(x)-> if (matrix{funct1(x)}*(MatA)==MatB) then P1#x else 0
--This is the list of standard monomials in the dual code
--If this list has |X|-dim_K(L_X) elements, then 
--the dual of L_X is monomial
dualevaluation1= set apply(0..#P1-1,funct2)-set{0}
--Now we compute the dual when the dual is not monomial 
Var7=(field)^**(#flatten entries basis(0,d,M))
Var8=(set{0})^**(#flatten entries basis(0,d,M))
Var9=apply(toList (Var7-Var8)/deepSplice,toList)
Var10=apply(apply(Var9,x->basis(0,d,M)*vector x),entries)
Var11=apply(toList set(apply(Var10,n->n#0)),
m->(leadCoefficient(m))^(-1)*m)
P=rsort(toList set Var11,MonomialOrder=>GRevLex)
b=apply(P,ps)
funct3=(n)->apply(B,x->evaluate(b#n,x))
MatC=matrix{apply(0..#B-1,n->{a^(q-1)})}
MatD=matrix{apply(0..#cL-1,n->{a^(q-1)-1})}
funct4=(x)-> if (matrix{funct3(x)}*(MatC)==MatD) then P#x else 0
--This is a list of generators of the algebraic dual:
dualevaluation= set apply(0..#P-1,funct4)-set{0}
--Next we compute a K-basis for the algebraic dual
--This computes the list of polynomials of a set 
--with maximum leading monomial
split=(a) -> set apply(0..#a-1, i-> if leadMonomial(a#i)==
leadMonomial(max(a)) then a#i else 0)-set{0}
--Iterating this function and taking max will give the K-basis
--for the algebraic dual
hhh=(a)->toList((set(apply(toList split(a),x->max(a)-
(leadCoefficient(max a)/leadCoefficient(x))*x))-set{0*t1})+
(set(a)-split(a)))
--Algorithm to compute a basis for a linear subspace 
--of K[t1,...,ts] of finite dimension. 
--This is a K-basis for the algebraic dual:
DDD=(A=toList dualevaluation; while #A>0 list 
max(A)/leadCoefficient(max(A)) do A=hhh(A))
--Next we compute the v-numbers and indicator functions for X
f=(n)->flatten flatten  degrees mingens(quotient(I,L#n)/I)
p=(n)->gens gb ideal(flatten mingens(quotient(I,L#n)/I))
minA=monomialIdeal(apply(0..#L-1,p))
vnumber0=min flatten degrees minA
g=(a)->toList(set a-set{0}) 
N=apply(apply(0..#L-1,f),g)
--This is the list of indicator functions for X:
toList apply(0..#L-1,n->p(n))
--Checking whether or not the homogenization I^h is Gorenstein
R=Fq[t1,t2,t3,u,MonomialOrder=>GRevLex]
J=sub(I,R)
L=ideal(homogenize(gens gb J,u))
HS=hilbertSeries(L)
--This gives the h-vector of the homogenization I^h of I
reduceHilbert(HS)
--This computes the minimal graded resolution of S/I^h
--that is used to determine whether or not 
--the homogenization I^h is Gorenstein
res(coker gens gb L)
\end{verbatim}
\end{procedure}

\end{appendix}


\section*{Acknowledgments} 
Computations with \textit{Macaulay}$2$ \cite{mac2} were important to 
give examples and to have a better understanding of the dual of evaluation
codes. 

\bibliographystyle{plain}

\end{document}